\documentclass[12pt]{amsart}
\usepackage{amscd,amssymb,amsthm,verbatim}
\newcommand{\C}{{\mathbb C}}

\newcommand{\const}{\operatorname{const.}}

\newcommand{\diam}{\operatorname{diam}}

\newcommand{\Dom}{\operatorname{Dom}}
\newcommand{\dvol}{\operatorname{dvol}}

\newcommand{\HH}{\operatorname{H}}

\newcommand{\IC}{\operatorname{IC}}
\newcommand{\Id}{\operatorname{Id}}
\newcommand{\IH}{\operatorname{IH}}

\newcommand{\Image}{\operatorname{Im}}

\newcommand{\Ker}{\operatorname{Ker}}

\newcommand{\Lip}{\operatorname{Lip}}

\newcommand{\R}{{\mathbb R}}

\newcommand{\Ric}{\operatorname{Ric}}

\newcommand{\supp}{\operatorname{supp}}

\newcommand{\vol}{\operatorname{vol}}
\newcommand{\Z}{{\mathbb Z}}

\numberwithin{equation}{section}
\theoremstyle{plain}

\newtheorem{lemma}[equation]{Lemma}
\newtheorem{theorem}[equation]{Theorem}
\newtheorem{proposition}[equation]{Proposition}
\newtheorem{corollary}[equation]{Corollary}

\theoremstyle{remark}

\newtheorem{remark}[equation]{Remark}
\newtheorem{example}[equation]{Example}
\usepackage{graphicx}
\input{amssym.def}
\input{amssym}
\input{epsf}
\usepackage{a4wide}

\begin{document}

\title[Eigenvalue estimates and differential form Laplacians]
{Eigenvalue estimates and differential form Laplacians on
Alexandrov spaces}

\author{John Lott}
\address{Department of Mathematics\\
University of California, Berkeley\\
Berkeley, CA  94720-3840\\
USA} \email{lott@berkeley.edu}

\thanks{Research partially supported by NSF grant
DMS-1510192}
\date{January 5, 2018}

\begin{abstract}
  We give upper bounds on the eigenvalues of the differential form
  Laplacian on a compact Riemannian manifold. The proof uses
  Alexandrov spaces with curvature bounded below.  We also construct
  differential form Laplacians on Alexandrov spaces.  Under a local
  biLipschitz assumption on the Alexandrov space,
  which is conjecturally always satisfied,
  we show that the differential form Laplacian has a compact resolvent.
  We identify its kernel with an intersection homology group.
\end{abstract}

\maketitle


\section{Introduction} \label{sec1}

For a closed connected Riemannian manifold, let $\lambda_k$ denote the
$k^{th}$ positive eigenvalue of the (nonnegative) function Laplacian,
counted with multiplicity.
In 1975, S.-Y. Cheng proved the following upper bound on the eigenvalues.

\begin{theorem} (Cheng \cite{Cheng (1975)}) \label{thm1.1}
  There is a function $c : \Z^+ \times \R \rightarrow \R^+$ with the
  following property. Given $n \in \Z^+$, $D \in \R^+$ and $K \in \R$,
  if $M$ is an $n$-dimensional closed connected
  Riemannian manifold with diameter $D$,
  and $\Ric_M \ge K g_M$, then
\begin{equation} \label{1.2}
  \lambda_k \le c(n, KD^2) \frac{k^2}{D^2}.
  \end{equation}
  \end{theorem}

We  give extensions of Cheng's result to the differential form Laplacian.  Let
$\lambda_{k,p}$ denote the $k^{th}$ positive eigenvalue of the
Hodge Laplacian $dd^* + d^*d$ on $p$-forms, counted with multiplicity.
First, we assume a lower volume bound.

\begin{theorem} \label{thm1.7}
  There is a function $C_1 : \Z^+ \times \R \times \R^+
  \rightarrow \R^+$ with the
  following property. Given $n \in \Z^+$, $D \in \R^+$, $K \in \R$
  and $v \in \R^+$,
  if $M$ is an $n$-dimensional closed connected
  Riemannian manifold with diameter $D$,
  sectional curvatures bounded below by $K$ and volume bounded below by
  $v$, then for all $p \in
  [0,n]$ we have
\begin{equation} \label{1.8}
  \lambda_{k,p} \le C_1(n, KD^2, vD^{-n}) \frac{k^{\frac{2}{n}}}{D^2}.
  \end{equation}
  \end{theorem}

To remove the lower volume bound,
we use the notion of a strainer
\cite[\S 5]{Burago-Gromov-Perelman (1992)}.
Given a complete Riemannian manifold
(or Alexandrov space) with curvature bounded below by $K \in \R$, and an
integer
$s > 0$, an $s$-strainer of quality $\delta$
and size $S$ at a point $m$ consists of points $\{a_i, b_i\}_{i=1}^s$
with $d(m,a_i) = d(m, b_i) = S$
so that
\begin{align}
\tilde{\measuredangle} a_i m b_i > \pi - \delta, \: \: \: \: \: \: & 
\tilde{\measuredangle} a_i m a_j > \frac{\pi}{2} - \delta, \\
\tilde{\measuredangle} a_i m b_j > \frac{\pi}{2} - \delta,
 \: \: \: \: \: \: &
\tilde{\measuredangle} b_i m b_j > \frac{\pi}{2} - \delta, \notag
\end{align}
whenever $i \neq j$. Here $\tilde{\measuredangle}$ is the comparison
angle at $m$, relative to the  model space of constant sectional curvature $K$.

\begin{theorem} \label{thm1.3}
  There is a function $C_2 : \Z^+ \times \R \rightarrow \R^+$ with the
  following property. Given $n \in \Z^+$ and $K \in \R$,
  let $M$ be an $n$-dimensional closed connected
  Riemannian manifold with
  sectional curvatures bounded below by $K$.  Suppose that there is
  some $m \in M$ with an $s$-strainer of quality $\frac{1}{10}$ and
  size $S$, where $1 \le s \le n$. Then for all $p \in
  [0,s]$, we have
  \begin{equation} \label{1.4}
    \lambda_{k,p} \le C_2(n, KS^2) \frac{k^{\frac{2}{s}}}{S^2}.
\end{equation}
  \end{theorem}

\begin{corollary} 
  Given $n \in \Z^+$, there is some $\kappa_n < \infty$ 
 with the
  following property.
  Let $M$ be an $n$-dimensional closed connected
  Riemannian manifold with nonnegative sectional curvature.
Suppose that there is
  some $m \in M$ with an $s$-strainer of quality $\frac{1}{10}$ and
  size $S$, where $1 \le s \le n$. Then for all $p \in
  [0,s]$, we have
  \begin{equation}
    \lambda_{k,p} \le \kappa_n \frac{k^{\frac{2}{s}}}{S^2}.
  \end{equation}
  \end{corollary}

\begin{remark}
  The choice of $\frac{1}{10}$ for the quality of the strainer is arbitrary.
  In the proof of Theorem \ref{thm1.3} we actually get upper
  eigenvalue bounds for the Laplacian on $\Omega^p(M)/\Ker(d)$ when
  $p < s$, and for the Laplacian on $\Image(d) \subset \Omega^s(M)$.

  Theorem \ref{thm1.3} implies Theorem \ref{thm1.1}
  under the stronger assumption of a
  lower sectional curvature bound, by taking $s = 1$ and $S = \frac{D}{2}$. 
  One would not expect to be able to control
   eigenvalues of the $p$-form Laplacian from a lower Ricci curvature bound if
  $p \notin \{0,1,n-1,n\}$.
\end{remark}

\begin{remark} \label{rem1.11}
  Theorem \ref{thm1.7} actually follows from Theorem \ref{thm1.3}. Under the
  hypotheses of Theorem \ref{thm1.7}, after rescaling the diameter to be $1$,
  convergence theory implies that there is some point with an
  $n$-strainer of quality
  $\frac{1}{10}$ and a certain size.
\end{remark}

\begin{remark}
  To get
  upper eigenvalue bounds for the $p$-form Laplacian, Theorem \ref{thm1.3}
  has an assumption about the existence of 
  a $p$-strainer at some point. The need for some such assumption
  can be seen
  by taking $M = X \times S^N$, where $X$ is a closed Riemannian
  manifold, and shrinking the $S^N$-factor. If
  $\dim(X) < p$  and $N > p$
  then the eigenvalues of the $p$-form Laplacian on $M$
  go to infinity.  
  As the sphere shrinks, there is clearly no $p$-strainer on $M$ of quality
  $\frac{1}{10}$ whose size is uniformly bounded below.
  
  Taking $s = \dim(X)$, this example also shows the sharpness of the
  exponent
  $\frac{2}{s}$ in (\ref{1.4}), which corresponds to Weyl-type
  asymptotics on an $s$-dimensional manifold.
\end{remark}

The constant $c$ in Theorem \ref{thm1.1} can be made explicit.
The constants $C_1$ and $C_2$ in
Theorems \ref{thm1.7} and \ref{thm1.3} are not explicit.  The reason is that
the proofs of Theorems \ref{thm1.7} and \ref{thm1.3} are by a
contradiction argument.
One common feature of all the proofs is the use of a 
minmax argument on an appropriate class of test functions or test forms.
To prove Theorem \ref{thm1.1}, Cheng
transplanted functions from a model space, using the exponential map
from a point. Dodziuk extended Cheng's result, by transplanting
differential forms from a model space, to prove an analog of
Theorem \ref{thm1.7} under the stronger
  assumptions of a double sided curvature bound and a lower bound on the
  injectivity radius
  \cite{Dodziuk (1982)}.
We instead pullback differential forms from an
Alexandrov space (with curvature bounded below)
that arises in the contradiction argument.

This leads to the question of whether differential
form Laplacians make sense for Alexandrov spaces.
(The function Laplacian on an Alexandrov space was studied in
\cite{Kuwae-Machigashira-Shioya (2001),Shioya (2001)}.)
To see some of the issues involved, note that on a smooth
Riemannian
manifold, when written in local coordinates, the differential
form Laplacian $dd^* + d^* d$ involves two derivatives of the
metric tensor.  An Alexandrov space has a dense open set with
the structure of a Riemannian Lipschitz manifold, meaning
in particular
that there is a Riemannian metric whose components, in local coordinates,
are in $L^\infty_{loc}$ \cite{Perelman}. Hence defining $dd^* + d^* d$
directly on an Alexandrov space does not look promising.

Instead of
trying to directly define the differential form Laplacian as an operator, one
could try to define the putative spectrum. On a smooth closed
Riemannian manifold $M$,
the minmax formula says that
\begin{equation} \label{1.13}
  \lambda_{k,p} = \inf_V \sup_{\omega \in V, \omega \neq 0}
  \frac{|d\omega|_{L^2}^2 + |d^* \omega|_{L^2}^2
  }{
|\omega|_{L^2}^2    },
\end{equation}
where $V$ ranges over $k$-dimensional subspaces of $\Omega^p(M)$.
In local coordinates, $d^* \omega$ involves first
derivatives of the metric tensor.  On an Alexandrov space,
one knows that the first
derivatives of the metric components
exist as measures \cite{Perelman}, but this is not
enough to make sense of (\ref{1.13}).

To gain another derivative, we use the observation,
essentially due to
Cheeger and Dodziuk \cite{Dodziuk (1982)}, 
that the minmax equation (\ref{1.13}) takes a nicer form if we
look instead
at the Laplacian $\triangle_*$ on $\Omega^*(M)/\overline{\Image(d)}$.
For this Laplacian, the minmax equation becomes
\begin{equation} \label{1.14}
  \lambda_{k,p} = \inf_V \sup_{\omega \in V, \omega \neq 0}
  \frac{|d\omega|_{L^2}^2
  }{
|\omega|_{L^2}^2    },
\end{equation}
where $V$ now ranges over $k$-dimensional subspaces of
$\Omega^p(M)/\overline{\Image(d)}$. The right-hand side of (\ref{1.14})
does not involve any derivatives of the metric tensor.
Using the Hodge decomposition
and the isomorphism 
$\Omega^*(M)/\overline{\Image(d)} \cong
\Ker(dd^* + d^*d) \oplus \overline{\Image(d^*)}$, the
spectrum of $\triangle_*$ is the same as the
spectrum of $dd^* + d^* d$, with the multiplicities related by
a factor of at most two.

For this reason, in making sense of a differential form Laplacian on an
Alexandrov space, we only consider an analog of the Laplacian $\triangle_*$
on $\Omega^*(M)/\overline{\Image(d)}$.

\begin{theorem} \label{thm1.15}
  If $X$ is a compact Alexandrov space then there is a well-defined
  nonnegative self-adjoint differential form Laplacian $\triangle_*$.
  When $X$ is a smooth Riemannian manifold $M$ with
  (possibly empty) convex boundary,
  the operator $\triangle_*$ becomes the usual Hodge
  Laplacian on $\Omega^*(M)/\overline{\Image(d)}$
  with relative (Dirichlet) boundary conditions.
\end{theorem}

We prove Theorem \ref{thm1.7} in the generality of compact Alexandrov spaces.
The use of $\Omega^*/\Ker(d)$ is key in proving
Theorems \ref{thm1.7} and \ref{thm1.3}. Their proofs do not need
the existence of the differential form Laplacian on the limit space,
but rather the existence of differential forms.

We construct $\triangle_*$ more generally for compact metric
spaces $X$ that have an open subset, of full Hausdorff measure, with
the structure of a Riemannian Lipschitz manifold.
The basic analytic
property of $\triangle_*$ that one would like to show is that
$(I + \triangle_*)^{-1}$
is compact; this implies discreteness of the spectrum of
$\triangle_*$.  In order to show that
$(I + \triangle_*)^{-1}$ is compact,
it is necessary to make an additional assumption about
$X$. To motivate this assumption, we recall that in
a finite dimensional
Alexandrov space $X$, any $x \in X$ has
a neighborhood that is homeomorphic to the truncated tangent cone $T^1_xX$
\cite{Kapovitch (2007),Perelman2}. It seems likely that
 any $x \in X$ has
 a neighborhood that is biLipschitz homeomorphic to $T^1_x X$;
 this has been claimed, although no proof is available.
Based on this,
 we consider a class ${\mathcal C}_*$ of
 compact metric spaces that are Lipschitz analogs of the topological
 multiconical spaces (MCS) introduced in
 \cite{Siebenmann (1972)} and used in \cite{Kapovitch (2007),Perelman2}.
 First, ${\mathcal C}_0$ consists of finite metric spaces. Inductively, if
 $X \in {\mathcal C}_n$ with $n \ge 1$
 then any point in $X$ has a neighborhood that
 is biLipschitz homeomorphic to the truncated open metric cone over some
 element of ${\mathcal C}_{n-1}$ with diameter at most $\pi$.
 Conjecturally, any $n$-dimensional compact
 Alexandrov space is an element of ${\mathcal C}_n$.  (If one is just
 interested in 
 Alexandrov spaces then one can just start with elements of ${\mathcal C}_0$
 consisting of one or two points.  For boundaryless
Alexandrov spaces, one can just start with elements of ${\mathcal C}_0$
consisting of two points.) Examples of elements of ${\mathcal C}_*$ come
from quotients of smooth closed Riemannian manifolds by compact groups
of isometries. Other examples come from compact stratified spaces with
iterated cone-edge Riemannian metrics.

 \begin{theorem} \label{thm1.16}
   (1) If $X \in {\mathcal C}_n$ then $\Ker(\triangle_*)$ is isomorphic to
   $\IH^{GM}_{n-*}(X; {\mathcal O})$,
   the Goresky-MacPherson intersection homology of
   $X$ as defined using the upper middle perversity. \\
(2) If $X \in {\mathcal C}_n$ then $(I + \triangle_*)^{-1}$ is compact.
 \end{theorem}

 Here ${\mathcal O}$ is the orientation line bundle of the codimension-zero
 stratum of $X$.
 The upper middle perversity is the function $\overline{p} : \Z^{\ge 0}
 \rightarrow \Z$ given by $\overline{p}(0) = 0$ and $\overline{p}(j) =
 \left[ \frac{j-1}{2} \right]$ for $j \ge 1$. If $X$ is a conically
 stratified
 pseudomanifold (i.e. has no codimension-one strata) then it is
 well known that the $L^2$-cohomology of $X$ is related to the intersection
 (co)homology of $X$ with middle perversity.  If $X$ is allowed to have
 codimension-one strata, as in our case, then there are various notions
 of intersection (co)homology
 \cite{Friedman (2011)}. It is not immediately clear which one is the
 right one to describe $\Ker(\triangle_*)$.
 It turns out that the right
 one is the original Goresky-MacPherson intersection homology, extended
 to spaces with codimension-one strata, after an appropriate change of degree.

\begin{remark} \label{rem1.17}
 A finite dimensional Alexandrov space is locally Lipschitz contractible
 \cite{Mitsuishi-Yamaguchi (2014)}. Unfortunately, this does not help in
 proving Theorem \ref{thm1.16} for Alexandrov spaces that are not
 {\it a priori} in ${\mathcal C}_*$, due to boundedness issues.
\end{remark}

To summarize, we construct self-adjoint differential form Laplacians
for a class of compact metric spaces, that includes compact Alexandrov spaces.
For a more restricted class of compact metric spaces, that conjecturally
includes compact Alexandrov spaces, we show that the differential
form Laplacian has a compact resolvent.

The structure of the paper is as follows.  In Section \ref{sec2} we
construct the differential form Laplacian $\triangle_*$ on a class of
compact metric spaces.  Section \ref{sec3}
has the construction of a sheaf of certain locally-$L^2$ differential forms.
The eigenvalue bounds of Theorems
\ref{thm1.7} and \ref{thm1.3} are proven in Section \ref{sec4}.
In Section \ref{sec5} we consider the Lipschitz multiconical spaces
${\mathcal C}_*$ and prove Theorem \ref{thm1.16}.

I thank Vitali Kapovitch for consultations on Alexandrov spaces, and
Greg Friedman for consultations on intersection homology.
I thank Bruno Colbois for a correction to an earlier version
of the paper, and the referee for helpful remarks.

\section{Differential form Laplacian on an Alexandrov space} \label{sec2}

In this section we define differential form Laplacians on a class of
metric spaces that includes Alexandrov spaces.  
In Subsection \ref{subsec2.1} we consider a certain class of
test forms built out of Lipschitz functions.  Using them, in
Subsection \ref{subsec2.2} we define a complex of
$L^2$-forms.  Subsection \ref{subsec2.3} has the construction
of the differential form Laplacian.

For background material on Alexandrov spaces,
we refer to \cite[Chapter 10]{Burago-Burago-Ivanov (2001)}.

Let $(X, d_X)$ be a compact metric space with Hausdorff dimension $n$ and
finite $n$-dimensional Hausdorff mass.
If $X$ is disconnected then we assume
that the distance between points in distinct connected components is
infinity.  Suppose that there is an open subset
$X^* \subset X$, with full Hausdorff $n$-measure,
having the structure of an $n$-dimensional Riemannian
Lipschitz manifold. This means that $X^*$ has a manifold structure
with locally Lipschitz transition maps, and that it is equipped
with a Riemannian metric $g$ so that
in coordinate charts, $g$ and $g^{-1}$ are in 
$L^\infty_{loc}$. In addition, $d_X$ is compatible with the
metric $d_{X^*}$ on $X^*$ coming from $g$ \cite{DeCecco-Palmieri (1991)}, in the sense that
$d_X$ and $d_{X^*}$ coincide on some neighborhood of the diagonal in
$X^* \times X^*$.  In particular, if $F$ is a function with
  compact support in a coordinate neighborhood of $X^*$, and $F$ is Lipschitz
  in terms of the coordinates, then $F$ is a Lipschitz function on $X$.

\begin{example} \label{ex2.1}
Let $X$ be a compact Alexandrov space with curvature bounded below,
of Hausdorff dimension $n$. There is some $\delta_0 > 0$ with the
following property. Given $\delta \in
(0, \delta_0)$,
let $X^*_\delta$ be the set of points $x \in X$
such that the space of directions $\Sigma_x$ has
$(n-1)$-dimensional Hausdorff mass more than
$(1 - \delta)$ times that of $S^{n-1}$. 
Then
$X^*_\delta$ is an open convex subset of $X$ of full Hausdorff measure,
with the structure of a Riemannian Lipschitz
manifold \cite{Perelman}. In fact, there is a stronger DC-structure,
but this doesn't seem to matter for the considerations of this paper.
\end{example}

\subsection{Test forms} \label{subsec2.1}

For a smooth compact Riemannian manifold, 
we can define an operator $d$, on a dense subset of 
$L^2$-forms, by saying that $\omega \in \Dom(d)$ if the distributional
differential $d \omega$ is $L^2$. Here
the notion of distributional differential
uses smooth test forms. On our space $X$, it doesn't make
sense to talk about smooth forms. We will instead use
``test forms'' made from Lipschitz functions.
(Not to be confused with the test forms mentioned in the introduction.)
Let $\Omega^*_{\Lip}(X)$ be the graded-commutative
differential graded algebra generated by
$\{f_0 df_1 \ldots df_k\}$, 
where $f_i \in \Lip(X)$. In particular, an element of
$\Omega^k_{\Lip}(X)$ is a finite sum of expressions
$f_0 df_1 \ldots df_k$, and
$d(f_0 df_1 \ldots df_k) = 1 \cdot df_0
df_1 \ldots df_k$.
There is a relation $d(fg) = fdg + gdf$ for $f,g \in \Lip(X)$.
The elements of  $\Omega^*_{\Lip}(X)$ are also known as the
K\"ahler forms of the algebra $\Lip(X)$.
There is a homomorphism $\rho$ from 
$\Omega^*_{\Lip}(X)$ to
the locally-$L^\infty$ differential forms on 
$X^*$. This homomorphism need not be injective or surjective.

The test forms, or more precisely their image under $\rho$,
will actually be twisted by the flat orientation line bundle
${\mathcal O}$ of $X^*$. The fiber of
${\mathcal O}$ over $x \in X^*$ is $\HH^n(X^*, X^*-x; \R)$.
If $X^*$ is orientable then with a given orientation $c$,
the homomorphism $\rho_c$ to the ${\mathcal O}$-valued
differential forms on $X^*$ can be identified with the $\rho$ of before.
If $c^\prime$ is a different orientation then in the applications,
$\rho_{c^\prime}(\omega^\prime)$
will be equivalent to the result of changing $\rho_c(\omega^\prime)$ by a sign on the
components of $X^*$ where $c^\prime$ differs from $c$.
We write $\Omega^*_{\Lip}(X; {\mathcal O})$ for the elements of
$\Omega^*_{\Lip}(X)$ when we consider them to be twisted by ${\mathcal O}$.

If $X^*$ is not orientable then we only consider the case when
$X$ is a boundaryless Alexandrov space.
There is a notion of an orientation cover $\widehat{X}$ of $X$
\cite{Harvey-Searle (2016)}. It is also an Alexandrov space and is
equipped with
a $\Z_2$-action whose quotient is $X$.
Choose an orientation
on the connected space $\widehat{X}^*$.
Then $\Omega^*_{\Lip}(X; {\mathcal O}) \cong
\Omega^*_{\Lip}(\widehat{X}) \otimes_{\Z_2} \R$, where $\R$ has the
nontrivial representation of $\Z_2$.
(The papers 
\cite{Harvey-Searle (2016),Mitsuishi (2016)} discuss various
equivalent notions of orientability for Alexandrov spaces.)

In the rest of the paper we will only discuss the case when
$X^*$ is oriented, as the nonorientable case can be handled
by working $\Z_2$-equivariantly on $\widehat{X}$.

\begin{lemma} \label{newlemma}
  If $\omega^\prime \in \Omega^{n-1}_{\Lip}(X; {\mathcal O})$ is such that
  $\rho(\omega^\prime)$  and $\rho(d\omega^\prime)$ have compact support
  in $X^*$ then $\int_{X^*} \rho(d\omega^\prime) = 0$.  
\end{lemma}
\begin{proof}
  Put $K = \supp(\rho(\omega^\prime)) \cup
  \supp(\rho(d\omega^\prime))$.
  Let $\{U_i\}_{i=1}^N$ be relatively compact coordinate neighborhoods of $X^*$ that cover
  K. Let $\{\phi_i\}_{i=1}^N$ be nonnegative subordinate
  Lipschitz functions whose sum is one on $K$.
  Write $\omega^\prime$ as a finite sum $\sum_j f_0^j df_1^j \ldots df_{n-1}^j$.
  Then
  \begin{equation}
    \rho(\omega^\prime) = \sum_{i=1}^N \phi_i \rho(\omega^\prime) =
\sum_{i=1}^N \sum_j \rho( \phi_i f_0^j df_1^j \ldots df_{n-1}^j)
  \end{equation}
  and
      \begin{equation}
        \rho(d\omega^\prime) = \sum_{i=1}^N (d\phi_i \wedge \rho(\omega^\prime) +
        \phi_i \rho(d\omega^\prime)) =
\sum_{i=1}^N \sum_j \rho( d(\phi_i f_0^j) df_1^j \ldots df_{n-1}^j).
      \end{equation}
      Hence it suffices to prove the lemma with $f_0^j$ replaced by $\phi_i f_0^j$,
      for some fixed $i$.
    Choose a Lipschitz function $\eta_i$ with compact support in $U_i$ so that
    $\eta_i \phi_i = \phi_i$, i.e. $\eta_i$ is one on $\supp(\phi_i)$. Then
    \begin{align}
      & d(\phi_i f_0^j) d(\eta_i f_1^j) \ldots d(\eta_i f_{n-1}^j) = \\
      & ((d\phi_i) f_0^j + \phi_i df_0^j) \cdot
      ((d\eta_i) f_1^j + \eta_i df_1^j) \ldots ((d\eta_i) f_{n-1}^j + \eta_i df_{n-1}^j) \notag
    \end{align}
    and
    \begin{align}
      & \rho \left( d(\phi_i f_0^j) d(\eta_i f_1^j) \ldots d(\eta_i f_{n-1}^j) \right) = \\
      & ((d\phi_i) \rho(f_0^j) + \phi_i \rho(df_0^j)) \wedge
      ((d\eta_i) \rho(f_1^j) + \eta_i \rho(df_1^j)) \wedge \ldots \wedge
      (d\eta_i \wedge \rho(f_{n-1}^j) + \eta_i \rho(df_{n-1}^j)) = \notag \\ 
      & ((d\phi_i) \rho(f_0^j) + \phi_i \rho(df_0^j)) \wedge
       \rho(df_1^j) \wedge \ldots \wedge
       \rho(df_{n-1}^j) = \rho(d(\phi_i f_0^j) df_1^j \ldots df_{n-1}^j)). \notag 
    \end{align}
    Hence we can reduce the lemma to the case when each of 
    $f_0^j, f_1^j, \ldots, f_{n-1}^j$ has compact support in $U_i$.
    Using Euclidean coordinates
    on $U_i$, we can mollify the functions by
    convolution and take the mollification parameter to zero, to reduce to the case when 
    $f_0^j, f_1^j, \ldots, f_{n-1}^j$ are smooth functions of
the coordinates, in which case the lemma is evident.
\end{proof}

It is not immediately clear that $\rho(d\omega^\prime)$ is
determined by $\rho(\omega^\prime)$, but this turns out to be the case.
\begin{lemma} \label{unique}
  Given $\omega^\prime_1, \omega^\prime_2 \in \Omega^p_{\Lip}(X)$, if
  $\rho(\omega^\prime_1) = \rho(\omega^\prime_2)$ then 
$\rho(d\omega^\prime_1) = \rho(d\omega^\prime_2)$
\end{lemma}
\begin{proof}
  It is equivalent to show that if $\rho(\omega^\prime) = 0$ then $\rho(d\omega^\prime) = 0$.
  Suppose that $\rho(\omega^\prime) = 0$.
  For any $\omega \in \Omega^{n-p-1}_{\Lip}(X; {\mathcal O})$ such that $\rho(\omega)$ and
  $\rho(d\omega)$ have compact support in
  $X^*$, Lemma \ref{newlemma} implies that
  \begin{equation}
    0 = \int_{X^*} \rho(d(\omega^\prime \wedge \omega)) =
    \int_{X^*} \rho(d\omega^\prime) \wedge \rho(\omega).
        \end{equation}
  Let $U$ be a relatively compact coordinate neighborhood for $X^*$.
  Let $F$ be a Lipschitz
  function with support in $U$. Let $\phi$ be a Lipschitz
  function with support in $U$ that is identically one on
  $\supp(F)$. Put $\omega = F d(\phi x^{i_1}) \ldots d(\phi x^{i_p})$.
  Then $\rho(\omega) =
  F dx^{i_1} \wedge \ldots \wedge dx^{i_p}$.
Letting $\omega$ vary over such choices, the lemma follows.
\end{proof}

\begin{lemma}
  If $\omega^\prime \in \Omega^p_{\Lip}(X)$ then $\supp(\rho(d\omega^\prime)) \subset
  \supp(\rho(\omega^\prime))$.
\end{lemma}
\begin{proof}
  Suppose that $\supp(\rho(d\omega^\prime))$ is not contained in
  $\supp(\rho(\omega^\prime))$. Let
  $\omega \in \Omega^{n-p-1}(X; {\mathcal O})$ be such that
  $\rho(\omega)$ and $\rho(d\omega)$ have support in a relatively compact coordinate neighborhood
  of $X^* - \supp(\rho(\omega^\prime))$; such $\omega$ can be constructed as in the proof of
  Lemma \ref{unique}. Then
  \begin{align}
    0 = & \int_{X^*} \rho(d(\omega^\prime \wedge \omega)) =
    \int_{X^*} \left( \rho(d\omega^\prime) \wedge \rho(\omega)
    + (-1)^p \rho(\omega^\prime) \wedge \rho(d\omega) \right) \\
    = &
    \int_{X^*} \rho(d\omega^\prime) \wedge \rho(\omega). \notag
  \end{align}
  Letting $\omega$ vary over such choices gives a contradiction.
\end{proof}

For brevity, we will write
$\omega^\prime$ for $\rho(\omega^\prime)$ on $X^*$, and
$d\omega^\prime$ for $\rho(d\omega^\prime)$ on $X^*$. In what follows,
this should not cause confusion.

\subsection{$L^2$-complex} \label{subsec2.2}

Let $\Omega^*_{L^2}(X)$ be the $L^2$-differential forms on $X^*$.
There is a well-defined integration
$\int_{X*} : \Omega^n_{L^2}
\left( X; {\mathcal O} \right) \rightarrow \R$.

The map $\rho$ sends
$\Omega^*_{\Lip}(X)$ to $\Omega^*_{L^2}(X)$.

\begin{lemma} \label{dense}
  The image of $\rho$ is dense in $\Omega^*_{L^2}(X)$.
\end{lemma}
\begin{proof}
  As in the proof of Lemma \ref{unique}, let $U$ be a relatively
  compact coordinate neighborhood and let $F$ be a Lipschitz
  function with support in $U$.
  Then $F dx^{i_1} \wedge \ldots \wedge dx^{i_p}$ is in the
  image of $\rho$, 
  from which the lemma follows.
  \end{proof}

Let $\Omega^p_{L^2,d}(X)$ be the elements
$\omega \in \Omega^p_{L^2}(X)$ for which there is some
$\eta \in \Omega^{p+1}_{L^2}(X)$ so that for all
$\omega^\prime \in \Omega^{n-p-1}_{\Lip}(X; {\mathcal O})$, we have
\begin{equation} \label{2.2}
\int_{X^*} \left( d\omega^\prime \wedge \omega + 
(-1)^{n-p-1} \omega^\prime \wedge \eta \right) = 0.
\end{equation}
If such an $\eta$ exists then it is unique, and we put $d\omega = \eta$.
This defines a map $d : \Omega^p_{L^2,d}(X) \rightarrow
\Omega^{p+1}_{L^2}(X)$. (In the case when $X$ is a smooth closed Riemannian
manifold, the definition of $d$ is similar to how one defines
the maximal closed extension of the exterior derivative on smooth forms.)

\begin{example} \label{ex2.3}
If $X = [0,1]$ then $\Omega^0_{L^2,d}([0,1]) = \{f \in H^1([0,1]) :
f(0) = f(1) = 0\}$ and $\Omega^1_{L^2,d}([0,1]) = \Omega^1_{L^2}([0,1])$.
More generally, if $X$ is a smooth compact 
Riemannian manifold-with-boundary,
with boundary inclusion $i : \partial X \rightarrow X$, 
then an element $\omega$ of $\Omega^*_{L^2,d}(X)$ has a well-defined
restriction $i^* \omega$ in
$\Omega^*(\partial X)/\overline{\Image(d)}$
that vanishes. These are relative (Dirichlet) boundary
conditions.
\end{example}

\begin{remark} \label{rem2.6}
If we replace $X^*$ by an open subset of $X^*$ with full measure then
$\Omega^p_{L^2}(X)$ and $\Omega^p_{L^2,d}(X)$ do not change. (This would
not be the case if we required $\omega^\prime$ to have support in
a compact subset of $X^*$.)
As a consequence, $\Omega^p_{L^2}(X)$ and $\Omega^p_{L^2,d}(X)$ are
independent of the choice of $X^*$. Namely, if $X^*_1$ and $X^*_2$
are two different choices then in each case, the ensuing spaces 
$\Omega^p_{L^2}(X)$ and $\Omega^p_{L^2,d}(X)$ are the same as those
coming from $X^*_3 = X^*_1 \cap X^*_2$.

In particular, if $X$ is a compact $n$-dimensional
Alexandrov space, let $X^*_\delta$ be the subspace of Example
\ref{ex2.1}. If $\delta^\prime < \delta$ then
$X^*_{\delta^\prime}$ is an open subset of $X^*_\delta$ with
full measure.  Hence the notions of $\Omega^p_{L^2}(X)$ and $\Omega^p_{L^2,d}(X)$
are independent of $\delta$.
\end{remark}

\begin{lemma} \label{lem2.7}
The subspace $\Omega^p_{L^2,d}(X)$ is dense in
$\Omega^p_{L^2}(X)$.
\end{lemma}
\begin{proof}
It suffices to prove the corresponding statement for the elements of
$\Omega^p_{L^2}(X)$ with support in a fixed but arbitrary 
compact set $K \subset X^*$. 
If $\{\sigma_i \}_{i=1}^N$ are Lipschitz functions on $X$ and
$\{\omega_i \}_{i=1}^N$ are elements of $\Omega^p_{L^2,d}(X)$
then one can
check that $\sum_{i=1}^N \sigma_i \omega_i \in \Omega^p_{L^2,d}(X)$, with
$d \sum_{i=1}^N \sigma_i \omega_i =
\sum_{i=1}^N (d\sigma_i \wedge \omega_i +  \sigma_i d \omega_i)$.  
Using a covering of $K$ by relatively compact coordinate neighborhoods
of $X^*$, and nonnegative
subordinate Lipschitz functions whose sum is one on $K$,
we can reduce to the case when
$K$ is a closed ball in a fixed coordinate neighborhood.  Considering
forms with support in $K$ that are smooth with respect to the
given coordinates, the lemma follows.
\end{proof}

\begin{lemma} \label{lem2.8}
The operator $d : \Omega^{p}_{L^2,d}(X) \rightarrow \Omega^{p+1}_{L^2}(X)$
is closed.
\end{lemma}
\begin{proof}
Suppose that $\{\omega_i\}_{i=1}^\infty$ is a sequence in
$\Omega^{p}_{L^2,d}(X)$ so that there is a limit of pairs
$\lim_{i \rightarrow \infty} 
(\omega_i, d\omega_i) = (\omega_\infty, \eta_\infty)$ for some
$(\omega_\infty, \eta_\infty) \in 
\Omega^{p}_{L^2}(X) \oplus \Omega^{p+1}_{L^2}(X)$.
Replacing $\omega$ in (\ref{2.2}) by $\omega_i$ and passing to the limit
shows that $\omega_\infty \in \Omega^{p}_{L^2,d}(X)$ and
$\eta_\infty = d \omega_\infty$. This proves the lemma.
\end{proof}

\begin{lemma} \label{lem2.9}
The image of $d : \Omega^p_{L^2,d}(X) \rightarrow
\Omega^{p+1}_{L^2}(X)$ lies in $\Omega^{p+1}_{L^2,d}(X)$, and
$\overline{\Image(d)} \subset \Ker(d)$.
\end{lemma}
\begin{proof}
Given $\omega \in \Omega^p_{L^2,d}(X)$,
replacing $\omega^\prime$ in (\ref{2.2}) with $d\omega^\prime$ gives
\begin{equation} \label{2.10}
\int_{X^*} d\omega^\prime \wedge d\omega = 0
\end{equation}
for all $\omega^\prime \in \Omega_{\Lip}^{n-p-2}(X; {\mathcal O})$. 
It follows that $\Image(d) \subset \Omega^{p+1}_{L^2,d}(X)$ and
$d^2 = 0$. Since $d$ is a closed operator, $\Ker(d)$ is a closed subset
of $\Omega^{p+1}_{L^2}(X)$. Hence $\overline{\Image(d)} \subset \Ker(d)$.
\end{proof}

\subsection{Differential form Laplacian} \label{subsec2.3}

Define a quadratic form on $\Omega^{p}_{L^2,d}(X)/\overline{\Image(d)}
\subset \Omega^{p}_{L^2}(X)/\overline{\Image(d)}$ by
\begin{equation} \label{2.11}
Q(\omega) = \int_{X^*} \langle d\omega, d\omega \rangle \: \dvol_{X^*}.
\end{equation}
As $d$ is closed, it follows that $Q$ is a closed quadratic form in the
sense of \cite[Section VIII.6]{Reed-Simon (1980)}. There is a
corresponding self-adjoint operator $\triangle_p = d^* d$, densely
defined on $\Omega^{p}_{L^2}(X)/\overline{\Image(d)}$
\cite[Theorem VIII.15]{Reed-Simon (1980)},
\cite[Theorem X.25]{Reed-Simon (1975)}. Its domain is
\begin{equation} \label{2.12}
\Dom(\triangle_p) = \{\omega \in \Omega^{p}_{L^2,d}(X)/\overline{\Image(d)}
\: : \: d\omega \in \Dom(d^*) \}.
\end{equation}

We have isometric isomorphisms
\begin{equation} \label{2.13}
  \Omega^*_{L^2}(X) \cong
  \left( \Omega^{*}_{L^2}(X)/\overline{\Image(d)} \right)
  \oplus \overline{\Image(d)}
\end{equation}
and
\begin{equation} \label{2.14}
  \Omega^{*}_{L^2}(X)/\overline{\Image(d)} \cong
  \left( \Omega^{*}_{L^2}(X)/{\Ker(d)} \right) \oplus
  \left( \Ker(d)/\overline{\Image(d)} \right).
\end{equation}
Using the isomorphisms
\begin{equation} \label{2.15}
\Omega^{*}_{L^2}(X)/{\Ker(d)} \cong \left( \Ker(d) \right)^\perp
\end{equation}
and
\begin{equation} \label{2.16}
\Ker(d)/\overline{\Image(d)} \cong \Ker(\triangle),
    \end{equation}
we obtain an orthogonal decomposition
\begin{equation} \label{2.17}
  \Omega^*_{L^2}(X) \cong \left( \Ker(d) \right)^\perp \oplus \Ker(\triangle_*)
  \oplus
  \overline{\Image(d)}.
\end{equation}
We have already defined $\triangle_p$ on
\begin{equation} \label{2.18}
  \Omega^{*}_{L^2}(X)/\overline{\Image(d)} \cong
  \left( \Ker(d) \right)^\perp \oplus
  \Ker(\triangle_*).
\end{equation}
We can define the Laplacian on $\overline{\Image(d)} \subset
\Omega^*_{L^2}(X)$ by using the
isomorphism $d : \Omega^{*-1}_{L^2}(X)/\Ker(d) \rightarrow
\Image(d)$ to transfer the Laplacian from
$\Omega^{*-1}_{L^2}(X)/\Ker(d)$. In this way, there is a 
Laplacian $\triangle^{Hodge}_*$ on $\Omega^*_{L^2}(X)$ with a full
Hodge decomposition.  Using this Hodge decomposition,
spectral questions for $\triangle^{Hodge}_*$
reduce to spectral questions about the
Laplacian $\triangle_*$ on $\Omega^{*}_{L^2}(X)/\overline{\Image(d)}$.
In the rest of the paper, we mostly concentrate on the
latter Laplacian.

\begin{example} \label{ex2.19}
If $X$ is a compact Riemannian manifold-with-boundary then
$\triangle_p$ is the densely-defined Laplacian on the Hilbert space
$\Omega^{p}_{L^2}(X)/\overline{\Image(d)}$,
with relative (Dirichlet)
boundary conditions, e.g. $\Ker(\triangle_p) \cong \HH^p(X, \partial X; \R)$.
\end{example}

\begin{example} \label{ex2.20}
  We give an example in which $\triangle_p$ has an infinite dimensional
  kernel.  Start with the cone $(0,1] \times S^3$, equipped with the metric
    $g = dr^2 + r^2 g_{S^3}$. Glue a $4$-ball onto the $S^3$-boundary.
    Choose some point $m \in S^3$ and
    for each $i > 1$, perform a connected sum with a copy of $\C P^2$,
    with size comparable to $100^{-i}$, at the point $(i^{-1}, m)$ in
    the conical region.  Call the result $X^*$ and let
    $X$ be its $1$-point compactification.
Now
    $\Image \left( \HH^p_c(X^*) \rightarrow \HH^p(X^*) \right)$ injects
    into $\Ker(d : \Omega^p_{L^2,d}(X) \rightarrow
      \Omega^{p+1}_{L^2,d}(X))/
        \overline{\Image(d : \Omega^{p-1}_{L^2,d}(X) \rightarrow
          \Omega^p_{L^2,d}(X))}$; see
        \cite[Proposition 4]{Lott (1996)}, whose proof does not
        need completeness of $X^*$.  As
$\Image \left( \HH^2_c(X^*) \rightarrow \HH^2(X^*) \right)$
            is infinite dimensional, it follows that $\Ker(\triangle_2)$ is
            infinite dimensional.
\end{example}

\begin{remark}
  An alternative differential form Laplacian 
  can be defined using the closure of the differential on $\Omega^*_{\Lip}(X)$
  (sometimes called the minimal closed extension).
  Namely, say that an element $\omega \in \Omega^*_{L^2}(X)$ lies in
  $\Dom(d)$ if there is a sequence $\omega_i \in \Omega^*_{\Lip}(X)$
  such that $\lim_{i \rightarrow \infty} \omega_i
  = \omega$ in $\Omega^*_{L^2}(X)$, and
  $\lim_{i \rightarrow \infty} d\omega_i$ exists in
  $\Omega^{*+1}_{L^2}(X)$.
  If this is the case, put
  $d \omega = \lim_{i \rightarrow \infty} d\omega_i$; it is independent
  of the particular choice of $\{\omega_i\}_{i=1}^\infty$.
  Then $d$ is a closed operator and one can consider $d^* d$, acting on
  $\Omega^*_{L^2}(X)/\overline{\Image(d)}$.

  If $X$ is a compact Riemannian manifold-with-boundary then one recovers the
  differential form Laplacian on $\Omega^*_{L^2}(X)/\overline{\Image(d)}$,
  with absolute (Neumann) boundary conditions, this way.  The differential
  form Laplacian $\triangle_*$, as defined following (\ref{2.11}),
  is more convenient
  for the purposes of this paper, as will be seen in the proof of Theorem
  \ref{thm1.3}.
\end{remark}

\begin{remark} \label{rem2.21}
If $X$ is a smooth closed Riemannian manifold then there is a
Hodge Laplacian $\triangle^{Hodge}_p = dd^* + d^* d$ acting on
the $H^2$-regular $p$-forms.
The space of $H^2$-regular $p$-forms is independent of the particular
Riemannian metric.

If $X$ is a closed Riemannian Lipschitz manifold then
a Hodge Laplacian $\triangle^{Hodge}_*$, with dense domain in
$\Omega^*_{L^2}(X)$, was defined in
\cite{Teleman (1983)}.  However, there are some subtleties.  For
example, the corresponding quadratic form
\begin{equation} \label{2.22}
Q^{Hodge}(\omega) =
\int_X \left( \langle d \omega, d \omega \rangle +
\langle d^* \omega, d^* \omega \rangle \right) \: \dvol_X
\end{equation}
has domain
$\Dom(Q^{Hodge}) = \{ \omega \in \Omega^*_{L^2}(X) \: : \:
d\omega \in \Omega^{*+1}_{L^2}(X), d * \omega \in
\Omega^{n-*+1}_{L^2}(X) \}$. Due to the appearance of the Hodge duality
operator $*$ in $d*\omega$, the domain of $Q^{Hodge}$ definitely depends on the
precise $L^{\infty}_{loc}$-Riemannian metric used on $X$
\cite[p. 46]{Teleman (1983)}.

In contrast,
the quadratic form $Q$ of (\ref{2.11}) has domain
$\Omega^{*}_{L^2,d}(X)/\overline{\Image(d)}$ 
which, in the case of a closed
Riemannian Lipschitz manifold, is independent of the precise
Riemannian metric.
Hence the domain of $\triangle_*$ is also independent of the precise
Riemannian metric. 
This is one manifestation of the fact that
$\triangle_*$ has better biLipschitz properties than $\triangle^{Hodge}_*$.
\end{remark}

\section{$L^2$-sheaf} \label{sec3}

In this section we define a
sheaf $\Omega^*_{L^2_{loc},d}$ of differential graded complexes,
constructed from certain
locally-$L^2$ differential forms.  The result of this section will be
used in Section \ref{sec5}.

We continue with the setup of Section \ref{sec2}.
In particular, $X$ is a compact metric space and $X^* \subset X$ is an
open subset of full Hausdorff measure, with the structure of a
Riemannian Lipschitz manifold.

Given an open set $U \subset X$, let
$\Omega^*_{L^2_{loc}}(U)$ be the locally-$L^2$
differential forms on $U^* = X^* \cap U$.
Let $\Omega^p_{L^2_{loc},d}(U)$ be the elements
$\omega \in \Omega^p_{L^2_{loc}}(U)$ for which there is some
$\eta \in \Omega^{p+1}_{L^2_{loc}}(U)$ 
so that for all compact subsets $K \subset U$
and all
$\omega^\prime \in \Omega^{n-p-1}_{\Lip}(X; {\mathcal O})$ with support in $K$,
we have
\begin{equation} \label{3.1}
\int_{U^*} \left( d\omega^\prime \wedge \omega + 
(-1)^{n-p-1} \omega^\prime \wedge \eta \right) = 0.
\end{equation}
If such an $\eta$ exists then it is unique, and we put $d\omega = \eta$.
Note that an element of $\Omega^p_{L^2_{loc},d}(U)$ may not satisfy
relative boundary conditions in any sense.

\begin{lemma} \label{lem3.2}
The assignment $U \rightarrow \Omega^p_{L^2_{loc},d}(U)$ defines a sheaf
$\Omega^p_{L^2_{loc},d}$ on $X$.
\end{lemma}
\begin{proof}
Given $V \subset U$, there is clearly a restriction map
$r_{V,U} : \Omega^p_{L^2_{loc},d}(U) \rightarrow \Omega^p_{L^2_{loc},d}(V)$
that defines a presheaf. Let
$\{U_\alpha\}$ be an open covering of $U$. Given elements
$\omega_1, \omega_2 \in \Omega^p_{L^2_{loc},d}(U)$, if
$r_{U_\alpha, U}(\omega_1) = r_{U_\alpha, U}(\omega_2)$ for all $\alpha$, then
$\omega_1 = \omega_2$. 

Now suppose that 
$\omega_\alpha \in \Omega^p_{L^2_{loc},d}(U_\alpha)$ are such that
$r_{U_\alpha \cap U_\beta, U_\alpha}(\omega_\alpha) = 
r_{U_\alpha \cap U_\beta, U_\beta}(\omega_\beta)$
for all $\alpha$ and $\beta$. There is a unique
$\omega \in \Omega^p_{L^2_{loc}}(U)$ so that
$\omega_\alpha = r_{U_\alpha, U}(\omega)$ for all $\alpha$.
We have to show that $\omega \in \Omega^p_{L^2_{loc},d}(U)$.
Let $K$ be a compact subset of $U$. Then $K$ is covered
by a finite subset $\{U_i\}_{i=1}^N$ of the $U_\alpha$'s.
Let $\{\sigma_i\}_{i=1}^N$ be nonnegative subordinate Lipschitz functions whose
sum is one on $K$.
In particular, the support of $\sigma_i$
is a compact set $K_i \subset U_i$. For any 
$\omega^\prime \in \Omega^{n-p-1}_{\Lip}(X; {\mathcal O})$ with support in $K$,
since $\sigma_i \omega^\prime \in \Omega^{n-p-1}_{\Lip}(X; {\mathcal O})$
has support in
$U_i$, and $\omega_i \in \Omega^p_{L^2_{loc},d}(U_i)$, we have
\begin{equation} \label{3.3}
\int_{U_i^*} \left( d(\sigma_i \omega^\prime) \wedge \omega + 
(-1)^{n-p-1} \sigma_i \omega^\prime \wedge d\omega_i \right) = 0.
\end{equation}
Summing over $i$ gives
\begin{equation} \label{3.4}
\int_{U^*} \left( d\omega^\prime \wedge \omega + 
(-1)^{n-p-1} \omega^\prime \wedge \sum_{i=1}^N \sigma_i d\omega_i \right) = 0.
\end{equation}  
Thus $\omega \in \Omega^p_{L^2_{loc},d}(U)$, with
$d\omega = \sum_{i=1}^N \sigma_i d\omega_i$.
This proves the lemma.
\end{proof}

\section{Eigenvalue estimates} \label{sec4}

In this section we prove the eigenvalue estimates of
Theorems \ref{thm1.7} and \ref{thm1.3}. In both cases, the
proof is by a contradiction argument.

As mentioned in Remark \ref{rem1.11},
Theorem \ref{thm1.7} actually follows from Theorem \ref{thm1.3}.
As the proof of Theorem \ref{thm1.7} makes the strategy clearer,
we give it first, in Subsection \ref{subsec4.1}.
To get the upper eigenvalue bound,
we pullback forms from a nice subset of a limiting Alexandrov space
that has the same dimension $n$ as the manifolds.
In Subsection \ref{subsec4.2}
we prove Theorem \ref{thm1.3}.
The proof is more involved, in
that we need to pullback forms from an Alexandrov space that may
be of a lower dimension.  We show that there is a region
in the limiting Alexandrov space that is almost Euclidean,
and above which the approximating manifold has a fibration
structure with controlled geometry. We then pullback forms
as in the proof of Theorem \ref{thm1.7}.

\subsection{Noncollapsing case} \label{subsec4.1}

Let $X$ be a compact metric space of the type considered in
Section \ref{sec2}.
Given $k \in \Z^+$ and $p \in [0,n]$, put
\begin{equation} \label{4.1}
\lambda_{k,p}(X) = \inf_{V_{k,p}} \sup_{\psi \in V_{k,p}, \psi \neq 0} \frac{
\| d\psi \|^2
}{
\| \psi \|^2
},
\end{equation}
where $V_{k,p}$ ranges over $k$-dimensional subspaces of 
$\Omega^{p}_{L^2,d}(X)/\Ker(d)$.
If $\triangle_p$ has discrete spectrum
with finite multiplicities then $\lambda_{k,p}(X)$ is the
$k^{th}$ eigenvalue of $\triangle_p$ on
$\Omega^{p}_{L^2}(X)/\Ker(d)$,
counted with multiplicity
\cite[Theorem XIII.2]{Reed-Simon (1978)}.

\begin{lemma} \label{lem4.3}
If $X_1$ and $X_2$ are $C$-biLipschitz then
\begin{equation} \label{4.4}
C^{-2n-4p-2} \lambda_{k,p}(X_1) \le \lambda_{k,p}(X_2) \le
C^{2n+4p+2} \lambda_{k,p}(X_1)
\end{equation}
\end{lemma}
\begin{proof}
By reducing $X_1^*$ and $X_2^*$ if necessary,
we can assume that the biLipschitz map $h : X_1 \rightarrow X_2$ restricts
to a biLipschitz map between $X_1^*$ and $X_2^*$. Then on $X_1^*$, we have
$C^{-2} g_1 \le h^* g_2 \le C^2 g_1$.
There is a bounded pullback map $h^* : \Omega^*_{L^2}(X_2) \rightarrow
\Omega^*_{L^2}(X_1)$. One can check that
$h^*$ sends $\Omega^*_{L^2,d}(X_2)$ to $\Omega^*_{L^2,d}(X_1)$, with
$h^*d=dh^*$. Hence there is also a pullback
$h^* :  \Omega^p_{L^2,d}(X_2)/\Ker(d_{X_2}) \rightarrow
\Omega^p_{L^2,d}(X_1)/\Ker(d_{X_1})$.

Given
a $k$-dimensional subspace $V_{k,p}$ of
$\Omega^p_{L^2,d}(X_2)/\Ker(d_{X_2})$ and a nonzero element
$\widetilde{\psi} \in V_{k,p}$,
let ${\psi} \in \Omega^p_{L^2,d}(X_2)$ be a representative for
$\widetilde{\psi}$. Then
\begin{align} \label{4.6}
  \| d h^* \widetilde{\psi} \|^2 = 
\int_{X_1^*} \langle d h^* \psi, d h^* \psi \rangle_{g_1} \dvol_{X_1^*} = &
\int_{X_1^*} \langle h^* d\psi, h^* d\psi \rangle_{g_1} \dvol_{X_1^*} \\
\le & \: C^{2(p+1)} C^n
\int_{X_1^*} \langle h^* d\psi, h^* d\psi \rangle_{h^* g_2} h^* \dvol_{X_2^*}
\notag \\
= & \: C^{2(p+1)+n}
\int_{X_2^*} \langle d\psi, d\psi \rangle_{g_2} \dvol_{X_2^*} \notag \\
= &  C^{2(p+1)+n}  \| d \widetilde{\psi} \|^2 \notag
\end{align}
and
\begin{align} \label{4.7}
  \| h^* \widetilde{\psi} \|^2 = & \inf_{\tau \in \Ker(d_{X_1})}
  \| h^* {\psi} +  \tau \|_{L^2}^2 
  = \inf_{\sigma \in \Ker(d_{X_2})}
  \| h^* ({\psi} + \sigma) \|_{L^2}^2 \\
  = &  \inf_{\sigma \in \Ker(d_{X_2})} 
  \int_{X_1^*} \langle h^* ({\psi} + \sigma),
  h^* ({\psi} + \sigma) \rangle_{g_1} \dvol_{X_1^*} \notag \\
    \ge &
    \: C^{-2p} C^{-n}
    \inf_{\sigma \in \Ker(d_{X_2})}
    \int_{X_1^*} \langle h^* ({\psi} + \sigma), h^*
    ({\psi} + \sigma) \rangle_{h^* g_2} h^* \dvol_{X_2^*}  \notag \\
= & \: C^{-2p-n}
    \inf_{\sigma \in \Ker(d_{X_2})}
    \int_{X_2^*} \langle {\psi} + \sigma,
    {\psi} + \sigma \rangle_{g_2} \dvol_{X_2^*}
= C^{-2p-n} \| \widetilde{\psi} \|^2. \notag
\end{align}
It follows that
\begin{equation} \label{4.8}
C^{-2n-4p-2} \lambda_{k,p}(X_1) \le \lambda_{k,p}(X_2).
\end{equation}
Reversing the roles of $X_1$ and $X_2$ gives (\ref{4.4}).
This proves the lemma.
\end{proof}

We now prove a version of Theorem \ref{thm1.7} for compact Alexandrov spaces.

\begin{proposition} \label{prop4.9}
Given $n \in \Z^+$, $K \in \R$ and $v > 0$, there is some
$L = L(n,K,v) < \infty$
so that for any $n$-dimensional compact connected
Alexandrov space $X$ for which
\begin{enumerate}
\item the curvature of $X$ is bounded below by $K$,
\item $\diam(X) \le 1$ and
\item $\vol(X) \ge v$,
\end{enumerate}
and any $p \in [0,n-1]$ and $k \in \Z^+$, 
we have $\lambda_{k,p}(X) \le L k^{\frac{2}{n}}$.
Here $\lambda_{k,p}(X)$ is defined by (\ref{4.1}).
\end{proposition}
\begin{proof}
Suppose that the claim about $\lambda_{k,p}$ is not true.  Then there is a
sequence $\{X_i\}_{i=1}^\infty$ of $n$-dimensional Alexandrov spaces
and some $p \in [0, n-1]$
so that
\begin{enumerate}
\item the curvature of $X_i$ is bounded below by $K$,
\item $\diam(X_i) \le 1$ and
\item $\vol(X_i) \ge v$, but
\item $\lambda_{k_i,p}(X_i) \ge i k_i^{\frac{2}{n}}$
  for some $k_i \in \Z^+$.
\end{enumerate}

Let $c_i$ be the smallest integer such that $k_i \le 2^{nc_i}$.
Then
\begin{equation}
  \lambda_{2^{nc_i},p}(X_i) \ge
  \lambda_{k_i,p}(X_i) \ge i  k_i^{\frac{2}{n}} \ge
  i  \left( 2^{n(c_i-1)} \right)^{\frac{2}{n}} = \frac{i}{4} 4^{c_i}. 
  \end{equation}
Putting $\overline{k}_i = 2^{nc_i}$, we have
$\lambda_{\overline{k}_i,p}(X_i)
\ge \frac{1}{4} i  \overline{k}_i^{\frac{2}{n}}$.

After passing to a subsequence, we can assume that
$\lim_{i \rightarrow \infty} X_i = X_\infty$ in the Gromov-Hausdorff
topology, where $X_\infty$ is also an $n$-dimensional
Alexandrov space. Let $x_\infty$ be a regular point of $X_\infty$.
Then there is a neighborhood $U_\infty$ of $x_\infty$,
that is biLipschitz homeomorphic to $(0,1)^n \subset
\R^n$, such that for large $i$,
there are
\begin{itemize}
\item points $x_i \in X_i$,
\item neighborhoods $U_i$ of $x_i$, and
\item bijective maps $\sigma_i : U_i \rightarrow U_\infty$
\end{itemize}
that are uniformly biLipschitz 
\cite[Theorem 9.8]{Burago-Gromov-Perelman (1992)}.

Letting $h_i$ be the composition of $\sigma_i$ with the biLipschitz
homeomorphism from $U_\infty$ to $(0,1)^n$, the maps
$h_i : U_i \rightarrow (0,1)^n$ are $\Lambda$-biLipschitz for
some $\Lambda < \infty$ independent of $i$.

Let $\psi$ be a smooth compactly supported $p$-form on
$(0,1)^n$ with $\int_{(0,1)^n} |\psi|^2 \: d^nx = 1$ and
$d\psi \neq 0$. Put $E_\psi = \int_{(0,1)^n} |d\psi|^2 \: d^nx$.

Let $\Ker(d_{(0,1)^n})$ be the forms
$\mu \in \Omega^{p}_{L^2}((0,1)^n)$ such that
\begin{equation} \label{newint}
\int_{(0,1)^n} d\omega^\prime \wedge \mu = 0
\end{equation}
for all compactly supported
$\omega^\prime \in \Omega^{n-p-1}_{\Lip}((0,1)^n)$.
Then $\Ker(d_{(0,1)^n})$ is closed in $\Omega^{p}_{L^2}((0,1)^n)$.

\begin{lemma} \label{morelem}
  $\psi$ does not lie in $\Ker(d_{(0,1)^n})$.
  \end{lemma}
\begin{proof}
  By assumption, $d\psi \neq 0$ as a smooth form. Hence we can find some
  $\omega^\prime \in \Omega^{n-p-1}_{\Lip}((0,1)^n)$ so that
  $\int_{(0,1)^n} d\omega^\prime \wedge \psi = (-1)^{n-p}
  \int_{(0,1)^n} \omega^\prime \wedge d\psi\neq 0$.
\end{proof}

With reference to Lemma \ref{morelem},
let $N_\psi$ be the square of the norm of the image of $\psi$ in
$\Omega^p_{L^2}((0,1)^n)/\Ker(d_{(0,1)^n})$.

Let $R_i : (0, 2^{-c_i})^n \rightarrow (0,1)^n$ be multiplication by
$2^{c_i}$. Under rescaling,
$\int_{(0, 2^{-c_i})^n} |dR_i^* \psi|^2 \: dx^n =
2^{c_i(2p+2-n)} E_\psi$ and the square norm of $R_i^* \psi$ in
$\Omega^p_{L^2}((0, 2^{-c_i})^n)/\Ker(d)$ is
$2^{c_i(2p-n)} N_\psi$.

There are $\overline{k}_i = 2^{nc_i}$ disjoint boxes
$\{B_j\}_{j=1}^{\overline{k}_i}$ in $(0,1)^n$, each congruent to
$(0, 2^{-c_i})^n$.
Let $\psi_j$ be the translate of
$R_i^* \psi$ to $B_j$.
Let $V^\infty_i$ be the span of
$\{\psi_j\}_{j=1}^{\overline{k}_i}$ in $\Omega^p_{L^2}((0,1)^n)$,
let ${h_i^* V^\infty_i}$ denote the extension by zero 
from $\Omega^p_{L^2,d}(U_i)$ to 
$\Omega^p_{L^2,d}(X_i)$ of the pullback, and
let $\widetilde{h_i^* V^\infty_i}$ denote the 
image of $h_i^* V^\infty_i$ in 
$\Omega^p_{L^2,d}(X_i)/\Ker(d_{X_i})$. We claim that
$\widetilde{h_i^* V^\infty_i}$ is $\overline{k}_i$-dimensional.
To see this, if there is a relation
$\sum_{j=1}^{\overline{k}_i} \alpha_j h_i^* \psi_j \in \Ker(d_{X_i})$
then
\begin{equation}
h_i^* \left( \sum_{j=1}^{\overline{k}_i} \alpha_j d \psi_j \right) =
d \left( \sum_{j=1}^{\overline{k}_i} \alpha_j h_i^* \psi_j \right) = 0,
\end{equation}
and hence
$\sum_{j=1}^{\overline{k}_i} \alpha_j d \psi_j = 0$, which implies that each $\alpha_j$
vanishes.

If $\eta_i = \sum_{j=1}^{\overline{k}_i} \beta_j h_i^* \psi_j$ is a
nonzero element of $h_i^* V^\infty_i$ then
as in the proof of (\ref{4.6}),
for large $i$, we have
\begin{align} \label{new1}
  \int_{X_i^*} \langle d\eta_i, d\eta_i \rangle_{g_i} \dvol_{X_i^*} \le &
  \Lambda^{2p+2+n} \sum_{j=1}^{\overline{k}_i} |\beta_j|^2
  \int_{(0,1)^n} \langle d\psi_j, d\psi_j \rangle_{(0,1)^n}
  \dvol_{(0,1)^n} \\
  = & 2^{c_i(2p+2-n)} E_\psi
  \Lambda^{2p+2+n} \sum_{j=1}^{\overline{k}_i} |\beta_j|^2. \notag
\end{align}

Let $\widetilde{\eta}_i$ be the class of $\eta_i$ in
$\widetilde{h_i^* V^\infty_i}$. 
We claim that
\begin{equation} \label{new2}
\| \widetilde{\eta}_i \|^2 \ge 
\frac14 2^{c_i(2p-n)} N_\psi \Lambda^{-2p-n}
\sum_{j=1}^{\overline{k}_i} |\beta_j|^2.
\end{equation}
To see this, suppose that
$\| \widetilde{\eta}_i \|^2 < \frac14 2^{c_i(2p-n)} N_\psi \Lambda^{-2p-n}
  \sum_{j=1}^{\overline{k}_i} |\beta_j|^2$. Then there is some
$\sigma \in \Ker(d_{X_i})$ so that
\begin{equation}
  \int_{X_i^*} \langle \eta_i + \sigma, \eta_i + \sigma \rangle_{g_i}
  \: \dvol_{X_i^*} < \frac12 2^{c_i(2p-n)} N_\psi \Lambda^{-2p-n}
  \sum_{j=1}^{\overline{k}_i} |\beta_j|^2
\end{equation}
As in (\ref{4.7}), it follows that
\begin{equation}
  \int_{(0,1)^n} \left\langle \sum_{j=1}^{\overline{k}_i}
  \beta_j \psi_j + (h_i^{-1})^* \sigma,
  \sum_{j=1}^{\overline{k}_i}
\beta_j \psi_j + (h_i^{-1})^* \sigma \right\rangle 
  \: \dvol_{(0,1)^n} < \frac12 2^{c_i(2p-n)} N_\psi 
  \sum_{j=1}^{\overline{k}_i} |\beta_j|^2
\end{equation}
However,
\begin{align}
&  \int_{(0,1)^n} \left\langle \sum_{j=1}^{\overline{k}_i}
  \beta_j \psi_j + (h_i^{-1})^* \sigma,
  \sum_{j=1}^{\overline{k}_i}
\beta_j \psi_j + (h_i^{-1})^* \sigma \right\rangle 
\: \dvol_{(0,1)^n} = \\
&  \sum_{j=1}^{\overline{k}_i}
  \int_{B_j} \left\langle
  \beta_j \psi_j + (h_i^{-1})^* \sigma,
  \beta_j \psi_j + (h_i^{-1})^* \sigma \right\rangle 
  \: \dvol_{B_j} \ge 2^{c_i(2p-n)} N_\psi
  \sum_{j=1}^{\overline{k}_i} |\beta_j|^2 . \notag
\end{align}
This is a contradiction.

Combining (\ref{new1}) and (\ref{new2}) gives
\begin{equation}
  \lambda_{\overline{k}_i,p}(X_i) \le 4 \cdot 2^{2c_i} \Lambda^{4p+2n+2} E_\psi
  N_{\psi}^{-1} =  4 \Lambda^{4p+2n+2} E_\psi
  N_{\psi}^{-1} \overline{k}_i^{\frac{2}{n}}.
\end{equation}
For large $i$, this contradicts the fact that
$\lambda_{\overline{k}_i,p}(X_i)
\ge \frac{1}{4} i  \overline{k}_i^{\frac{2}{n}}$.
This proves the proposition.
\end{proof}

\begin{corollary} \label{cor4.12}
Given $n \in \Z^+$, $K \in \R$ and $v > 0$, there is some
$L = L(n,K,v) < \infty$ with the following property.
Let $M$ be an $n$-dimensional compact connected
Riemannian manifold-with-boundary
whose boundary, if nonempty, is convex.
Suppose that
\begin{enumerate}
\item the sectional curvature of $M$ is bounded below by $K$,
\item $\diam(M) \le 1$ and
\item $\vol(M) \ge v$.
\end{enumerate}
Then
for any $p \in [0,n]$ and any $k \in \Z^+$,
the $k^{th}$ positive eigenvalue of the $p$-form Laplacian on $M$,
as defined with relative boundary conditions, satisfies
$\lambda_{k,p}(M) \le L k^{\frac{2}{n}}$.
\end{corollary}
\begin{proof}
  Proposition \ref{prop4.9} gives upper bounds on the eigenvalues
  of the Laplacian on $\Omega^{p}_{L^2}(M)/\Ker(d)$ for all $p \in [0, n-1]$.
  From the Hodge decomposition, we also get an upper bound on the
  eigenvalues of the Laplacian on $\Image(d) \subset \Omega^n_{L^2}(M)$.
  The corollary follows.
\end{proof}

Theorem \ref{thm1.7} follows from Corollary \ref{cor4.12}.

\subsection{General case} \label{subsec4.2}

We now prove Theorem \ref{thm1.3}.

\begin{proposition} \label{prop4.13}
Given $n \in \Z^+$ and $K \in \R$, there is some
$A = A(n,K) < \infty$ with the following property.
Let $M$ be an
$n$-dimensional closed connected Riemannian manifold for which
\begin{enumerate}
\item the sectional curvature of $M$ is bounded below by $K$ and
\item there is some point $m \in M$ with an $s$-strainer of quality
  $\frac{1}{10}$ and size $1$, where $1 \le s \le n$.
\end{enumerate}
Then for any $p \in [0,s-1]$ and $k \in \Z^+$, we have
$\lambda_{k,p}(M) \le A k^{\frac{2}{s}}$.
Here $\lambda_{k,p}(M)$ is the $k^{th}$ eigenvalue
of the Laplacian on $\Omega^p_{L^2}(M)/\Ker(d)$.
\end{proposition}
\begin{proof}
  Suppose that the claim about $\lambda_{k,p}$ is not true. 
  Then for
  some $n \in \Z^+$, some $s \in [1,n]$ and some $p \in [0, s-1]$, there is a
sequence $\{M_i\}_{i=1}^\infty$ of $n$-dimensional closed Riemannian
manifolds and numbers $k_i \in \Z^+$ so that
\begin{enumerate}
\item the curvature of $M_i$ is bounded below by $K$ and
\item there is some point $m_i \in M_i$ with an $s$-strainer of
  quality $\frac{1}{10}$ and size $1$, but
\item $\lambda_{k_i,p}(M_i) \ge i k_i^{\frac{2}{s}}$.
\end{enumerate}

Let $c_i$ be the smallest integer such that $k_i \le 2^{n! c_i}$.
Then
\begin{equation}
  \lambda_{2^{n!c_i},p}(M_i) \ge
  \lambda_{k_i,p}(M_i) \ge i  k_i^{\frac{2}{s}} \ge
  i  \left( 2^{n!(c_i-1)} \right)^{\frac{2}{s}}.
  \end{equation}
Putting $\overline{k}_i = 2^{n!c_i}$, we have
$\lambda_{\overline{k}_i,p}(M_i)
\ge \frac{i}{4^{\frac{n!}{s}}}  \overline{k}_i^{\frac{2}{s}}$.

After passing to a subsequence, we can assume that
$\lim_{i \rightarrow \infty} (M_i, m_i) = (X_\infty, x)$ in the
pointed Gromov-Hausdorff
topology, where $X_\infty$ is a complete
Alexandrov space, say of dimension $n_\infty$. From the strainer condition,
$n_\infty \ge s$.
Let $x_\infty$ be a regular point of $X_\infty$, say within
distance $\frac{1}{1000}$ from $x$.

Given $\delta > 0$,
we can find a $n_\infty$-strainer
$\{a_l, b_l\}_{l=1}^{n_\infty}$ of quality $\delta$
around $x_\infty$, say of size
$r_\delta$, with $\lim_{\delta \rightarrow 0} r_\delta = 0$. 

I thank Vitali Kapovitch for the proof of the next lemma.

\begin{lemma}
  If $\delta$ is sufficiently small then for large $i$, the following holds.
  There are an open subset $U_i$ of $M_i$, a
  closed Lipschitz manifold $Z_i$, a Lipschitz surjection
$\tau_i : U_i \rightarrow B(0, \delta r_\delta)$ 
  and a commutative diagram
\begin{equation}
  \begin{array}{ccc}
   \mbox{ } U_i & \stackrel{\alpha_i}{\longrightarrow} & Z_i \times B(0, \delta r_\delta) \\
    \tau_i \big\downarrow & & p_2 \big\downarrow \\
    B(0, \delta r_\delta) & \stackrel{\Id}{\longrightarrow} & B(0, \delta r_\delta),
  \end{array}
  \end{equation}
where $\alpha_i$ is a biLipschitz homeomorphism. Furthermore,
  $\tau_i$ is an almost metric submersion in the sense of
  \cite[p. 318]{Yamaguchi (1991)}.
\end{lemma}
\begin{proof}
Let $\widetilde{m}_i \in M_i$ be such that
$\{\widetilde{m}_i \}_{l=1}^{n_\infty}$
converges to $x_\infty$. Given $\delta > 0$,
for large $i$, let
$\{\widetilde{a}_{i,l}, \widetilde{b}_{i,l} \}_{l=1}^{n_\infty}$ be an
$n_\infty$-strainer of quality $2 \delta$
in $M_i$ such that as $i \rightarrow \infty$, 
$\{\widetilde{a}_{i,l} \}_{l=1}^{n_\infty}$
converges to $\{a_l\}_{l=1}^{n_\infty}$ and
$\{\widetilde{b}_{i,l} \}_{l=1}^{n_\infty}$
converges to $\{b_l\}_{l=1}^{n_\infty}$.
Define ${\gamma}_i : B(\widetilde{m}_i, r_{\delta})
\rightarrow \R^{n_\infty}$ by
${\gamma}_i(p_i) = \{d(\widetilde{m}_i, \widetilde{a}_{i,l}) -
d(p_i, \widetilde{a}_{i,l})
\}_{l=1}^{n_\infty}$.
  
 Consider the pointed Riemannian manifolds
  $\{(M_i, r_\delta^{-2} g_i, \widetilde{m}_i) \}_{i=1}^\infty$.
  The rescaled limit space
  $r_\delta^{-1} X_\infty$ has a strainer of size $1$ and quality
  $\delta$, centered at $x_\infty$.
  There is some $\rho > 0$  so that as $\delta \rightarrow 0$,
  the $\rho$-balls around $x_\infty$ in $r_\delta^{-1} X_\infty$
  converge in the pointed Gromov-Hausdorff topology to $B(0, \rho)
  \subset \R^{n_\infty}$.
  It follows that for any $\epsilon > 0$, there is a $\delta_0 > 0$ so that
if $\delta < \delta_0$ then
  for all large $i$, the map $r_\delta^{-1} \gamma_i : B(\widetilde{m}_i, 
  r_\delta) \rightarrow \R^{n_\infty}$ defines a pointed
  $\epsilon$-Gromov-Hausdorff approximation between
  $(\gamma_i^{-1}(B(0, \rho r_\delta)), \widetilde{m}_i)$ and $(B(0, \rho), 0)$,
  where $\gamma_i^{-1}(B(0, \rho r_\delta))$ has the restricted metric from
  $(M_i, r_\delta^{-2} g_i)$. 
  For an appropriate choice of $\epsilon$ and taking $\delta$ small enough,
  we can apply \cite{Yamaguchi (1991)} in the pointed setting to get the
  Lipschitz fibration, along
  with the almost metric submersion property. (In fact, the fibration in 
\cite{Yamaguchi (1991)} is $C^1$.) As $B(0, \delta r_\delta)$
  is contractible, the fibration structure is a product structure.
\end{proof}

Since $\tau_i$ is Lipschitz and surjective,
for almost all $u \in B(0, \delta r_\delta)$ and almost all
$u_i \in \tau_i^{-1}(u)$, the differential $d\tau_i : T_{u_i} M_i \rightarrow
\R^{n_\infty}$ is defined. Given such a point $u_i \in U_i$, 
put $V_{u_i} = \Ker((d\tau_i)_{u_i})$,
an $(n - n_\infty)$-dimensional subspace of $T_{u_i}U_i$. Put
$H_{u_i} = V_{u_i}^{\perp}$. The ``almost metric submersion''
property implies that if $\delta$ is small enough then for
all large $i$, we have
\begin{equation} \label{4.14}
 \frac12 | v| \le |(d\tau_i)_{u_i} v| \le 2 |v|.
\end{equation}
for all $v \in H_{u_i}$.
It follows that for $\eta \in \Lambda^p(T_{u}^* B(0, \delta r_\delta))$,
we have
\begin{equation} \label{4.15}
2^{-p} |\eta| \le |(d\tau_i)_{u_i}^* \eta| \le 2^{p} |\eta|.
\end{equation}

\begin{lemma} \label{lem4.16}
There is some $\widehat{C} < \infty$ so that
for all sufficiently large $i$ and all
$u, u^\prime \in B(0, \delta r_\delta)$, we have
\begin{equation} \label{4.17}
\widehat{C}^{-1} \vol(\tau_i^{-1}(u)) \le
\vol(\tau_i^{-1}(u^\prime)) \le
\widehat{C} \vol(\tau_i^{-1}(u)).
\end{equation}
\end{lemma}
\begin{proof}
The proof uses gradient flow. We sketch the argument, which is
similar to 
\cite[Pf. of Lemma 6.15]{Kapovitch (2007)}.
Write
$\tau_i = ( \xi_{i,1}, \ldots, \xi_{i, n_\infty})$
and $u = (a_1, \ldots, a_{n_\infty})$.
From the construction of $\tau_i$ using distance functions from
strainer points, the functions $\xi_{i, \cdot}$ are
quantitatively semiconcave, independent of $i$.
Put
$H_-(u) = \{\widehat{u} = (b_1, \ldots, b_{n_\infty}) \in \R^{n_\infty} :
b_l \le a_l
\mbox{ for }
1 \le l \le n_\infty \}$.
Given $m_i \in \tau_i^{-1}(u^\prime)$, we can
perform a gradient flow starting from $m_i$
with respect to the gradient of the
distance function from an appropriate point in $M_i$, for a controlled
amount of time, to ensure that the result lies in $\tau_i^{-1}(H_-(u))$.
Then we perform a gradient flow with respect to the gradient of
$F = \min(0, \xi_{i,1} - a_1, \ldots, \xi_{i, n_\infty} - a_{n_\infty})$.
After a controlled amount of time, the result of the flow lies in
$\tau_i^{-1}(u)$.
This gives a map
$L_{u^\prime, u} : \tau_i^{-1}(u^\prime) \rightarrow \tau_i^{-1}(u)$.
Using the control on the semiconcavity of the
distance functions and of $F$, along with the ensuing distortion 
bounds for gradient flow, we obtain a bound on the
Lipschitz constant of $L_{u^\prime, u}$ that is independent of
$u$, $u^\prime$ and $i$.   Replacing
$u^\prime$ by a point $u^{\prime \prime}$ moving along a line from
$u^\prime$ to $u$,
and performing the same construction, shows that if the
fibers are orientable then
$L_{u^\prime, u}$ has degree one; if the fibers are not orientable
then we pass to orientable double covers
of the fibers and apply the same arguments.
In all, from the Lipschitz bound on
$L_{u^\prime, u}$, we obtain a bound 
$\frac{\vol(\tau_i^{-1}(u))}{\vol(\tau_i^{-1}(u^\prime))} \le \widehat{C}$
with $\widehat{C}$ independent of $u$, $u^\prime$ and $i$, thereby giving the
first inequality in (\ref{4.17}). Reversing
the roles of $u$ and $u^\prime$ gives the second inequality in (\ref{4.17}).
\end{proof}

We continue with the proof of Proposition \ref{prop4.13}.
Let $\mu_\delta : U_\delta \rightarrow (0,1)^{n_\infty}$
be linear coordinates for a neighborhood $U_\delta$
of $0$ in $B(0, \delta r_\delta)$. We
redefine $U_i$ to be
$\tau_i^{-1}(U_\delta)$ and put $h_i = \mu_\delta \circ \tau_i :
U_i \rightarrow (0,1)^{n_\infty}$.

\begin{lemma} \label{chi}
  There exist $\Delta < \infty$ and, for sufficiently large $i$,
  a closed element $\chi_i \in
  \Omega^{n - n_\infty}_{\Lip}(U_i; {\mathcal O}_{U_i})$ so that
for all
$u \in (0,1)^{n_\infty}$, we have $\int_{h_i^{-1}(u)} \chi_i = 1$, and
  $\int_{h_i^{-1}(u)} |\chi|^2 \dvol_{h_i^{-1}(u)} \le \Delta \left( \inf_{u \in (0,1)^{n_\infty}} \vol(h_i^{-1}(u))
  \right)^{-1}$. 
  \end{lemma}
\begin{proof}
We can assume that
$h_i$ is similarly defined on a slightly larger open set containing
$U_i$, so that the fiber $h_i^{-1}(1, \ldots, 1)$ is well defined.
  Write $h_i = (h_{i,1}. \ldots, h_{i,n_\infty})$. Similarly to the proof of
  Lemma \ref{lem4.16}, we
  perform a gradient flow on $\overline{U_i}$ with respect to the
  gradient of $F = \min(0, h_{i,1} - 1, \ldots, h_{i,n_\infty}-1)$.
  After a controlled amount of time, the result of the flow lies in
  $h_i^{-1}(1, \ldots, 1)$.  This gives a deformation retraction
  $L : \overline{U_i} \rightarrow h_i^{-1}(1, \ldots, 1)$. Using
  the control on the semiconcavity of $F$, along with the ensuing
  distortion bounds for gradient flow, we obtain a bound on the
  Lipschitz constant of $F$. Let $\chi_i$ be the pullback under $L$ of the
  normalized volume density
\begin{equation}
\dvol_{h_i^{-1}(1, \ldots, 1)}/
  \vol(h_i^{-1}(1, \ldots, 1)) \in
  \Omega^{n-n_\infty}_{\Lip}(h_i^{-1}(1, \ldots, 1);
        {\mathcal O}_{h_i^{-1}(1, \ldots, 1)}).
        \end{equation}
  Then $\chi_i \in \Omega^{n-n_\infty}_{\Lip}(U_i; {\mathcal O}_{U_i})$.
  The bound on the Lipschitz constant
  of $F$, which is independent of $i$, gives a pointwise bound on $\chi_i$
  of the form $|\chi_i| \le \const
  \left( \vol(h_i^{-1}(1, \ldots, 1)) \right)^{-1}$.
  As
  \begin{align}
    \int_{h_i^{-1}(u)} |\chi_i|^2 \dvol_{h_i^{-1}(u)}
        \le & 
    \const \left( \vol(h_i^{-1}(1, \ldots, 1)) \right)^{-2}
    \sup_{u \in (0,1)^{n_\infty}} \vol(h_i^{-1}(u)) \notag \\
    \le & \const \widehat{C} \left( \inf_{u \in (0,1)^{n_\infty}}
    \vol(h_i^{-1}(u))
    \right)^{-1}, \notag
  \end{align}
  the lemma follows.
  \end{proof}

Define $\psi, E_\psi, \Ker(d_{(0,1)^{n_\infty}})$ and
$N_\psi$ as in the proof of
Proposition \ref{prop4.9}, with $n$ replaced by $n_\infty$.

Let $R_i : (0, 2^{- \frac{n!}{n_\infty} c_i})^{n_\infty} \rightarrow
(0,1)^{n_\infty}$ be
multiplication by
$2^{\frac{n!}{n_\infty} c_i}$. Under rescaling,
\begin{equation}
\int_{(0, 2^{-\frac{n!}{n_\infty} c_i})^{n_\infty}}
|dR_i^* \psi|^2 \: dx^{n_\infty} =
2^{\frac{n!}{n_\infty} c_i(2p+2-n_\infty)} E_\psi
\end{equation}
and the square norm of $R_i^* \psi$ in
$\Omega^p_{L^2}((0, 2^{-\frac{n!}{n_\infty}c_i})^{n_\infty})/
\Ker(d)$ is
$2^{\frac{n!}{n_\infty} c_i(2p-n_\infty)} N_\psi$.

There are $\overline{k}_i = 2^{n!c_i}$ disjoint boxes
$\{B_j\}_{j=1}^{\overline{k}_i}$ in $(0,1)^{n_\infty}$, each congruent to
$(0, 2^{-\frac{n!}{n_\infty}c_i})^{n_\infty}$.
Let $\psi_j$ be the translate of
$R_i^* \psi$ to $B_j$.
Let $V^\infty_i$ be the span of
$\{\psi_j\}_{j=1}^{\overline{k}_i}$ in $\Omega^p_{L^2} \left((0,1)^{n_\infty} \right)$.
Let ${h_i^* V^\infty_i}$ denote the extension by zero 
from $\Omega^p_{L^2}(U_i)$ to 
$\Omega^p_{L^2}(M_i)$ of the pullback.

\begin{lemma}
  For large $i$, $h_i^* V^\infty_i$ is a subspace of
$\Omega^p_{L^2,d}(M_i)$ that
  projects to a $\overline{k}_i$-dimensional 
  subspace $\widetilde{h_i^* V^\infty_i}$
  of $\Omega^p_{L^2,d}(M_i)/\Ker(d_{M_i})$.
\end{lemma}
\begin{proof}
  We claim that $h_i^* \psi_j$ is a well-defined
  element of $\Omega^p_{L^2,d}(M_i)$,
  with $d h_i^* \psi_j = h_i^* d\psi_j$.
    Let $\Omega^*_{\Lip,c} \left( (0,1)^{n_\infty} \right)$ denote the
  differential graded algebra constructed from compactly supported Lipschitz
  functions on $(0,1)^{n_\infty}$.
  There is a pullback $h_i^* : \Omega^*_{\Lip,c} \left( (0,1)^{n_\infty} \right)
  \rightarrow \Omega^*_{\Lip}(M_i)$ of differential graded algebras.
  Since $\psi_j$ is a smooth compactly supported $p$-form on $(0,1)^{n_\infty}$,
  it follows as in the proof of Lemma \ref{unique} that there is a
  $\psi_j^\prime \in \Omega^*_{\Lip,c} \left( (0,1)^{n_\infty} \right)$ so that
  $\rho(\psi_j^\prime) = \psi_j$. Then $h_i^* \psi_j = \rho(h_i^* \psi_j^\prime)$
  in $\Omega^p_{L^2}(M_i)$. Using Lemma \ref{newlemma}, one shows that
  $h_i^* \psi_j$ lies in $\Omega^p_{L^2,d}(M_i)$, with differential
  given by $d h_i^* \psi_j = \rho(d h_i^* \psi_j^\prime) =
  \rho(h_i^* d \psi_j^\prime) = h_i^* d\psi_j$.

  The lemma now
  follows as in the proof of Proposition \ref{prop4.9}.
  \end{proof}

For sufficiently small $\delta$ and all large $i$, 
the ratio
$\frac{(h_i)_* \dvol_{M_i}}{\dvol_{(0,1)^{n_\infty}}}$ is bounded above
by twice the function on
$(0,1)^{n_\infty}$ which, to a point $u \in (0,1)^{n_\infty}$,
assigns the volume of
the fiber $h_i^{-1}(u)$.
If $\eta_i = \sum_{j=1}^{\overline{k}_i} \beta_j h_i^* \psi_j$ is a
nonzero element of $h_i^* V^\infty_i$ then
using (\ref{4.15}), there is some $\Lambda < \infty$ such that
for large $i$, we have
\begin{align} \label{upper}
  \int_{M_i} \langle d\eta_i, d\eta_i \rangle_{g_i} \dvol_{M_i} \le 
  & \Lambda^{p+1}  \sum_{j=1}^{\overline{k}_i} |\beta_j|^2
  \int_{(0,1)^n} \langle d\psi_j, d\psi_j \rangle_{(0,1)^n}
\frac{(h_i)_* \dvol_{M_i}}{\dvol_{(0,1)^{n_\infty}}}
\dvol_{(0,1)^n} \\
\le & 2 \Lambda^{p+1}
2^{\frac{n!}{n_\infty} c_i(2p+2-n_\infty)} E_\psi
\cdot \sup_{u \in (0,1)^{n_\infty}} \vol(h_i^{-1}(u))
\cdot \sum_{j=1}^{\overline{k}_i} |\beta_j|^2. \notag
\end{align}

Let $\widetilde{\eta}_i$ be the class of $\eta_i$ in
$\widetilde{h_i^* V^\infty_i}$. 
We claim that
\begin{equation} \label{lower}
\| \widetilde{\eta}_i \|^2 \ge
\frac14 2^{\frac{n!}{n_\infty} c_i(2p-n_\infty)} \Delta^{-1} N_\psi
\left( \inf_{u \in (0,1)^{n_\infty}} \vol(h_i^{-1}(u)) \right)
\sum_{j=1}^{\overline{k}_i} |\beta_j|^2.
\end{equation}
To see this, suppose that
$\| \widetilde{\eta}_i \|^2 <
\frac14 2^{\frac{n!}{n_\infty} c_i(2p-n_\infty)} \Delta^{-1} N_\psi
\left( \inf_{u \in (0,1)^{n_\infty}} \vol(h_i^{-1}(u)) \right)
\sum_{j=1}^{\overline{k}_i} |\beta_j|^2$. Then there is some
$\sigma_i \in \Ker(d_{M_i})$ so that
\begin{equation} \label{needed}
  \int_{M_i} \langle \eta_i +  \sigma_i, \eta_i +  \sigma_i \rangle_{g_i}
  \: \dvol_{M_i} < \frac12
2^{\frac{n!}{n_\infty} c_i(2p-n_\infty)} \Delta^{-1} N_\psi
\left( \inf_{u \in (0,1)^{n_\infty}} \vol(h_i^{-1}(u)) \right)
\sum_{j=1}^{\overline{k}_i} |\beta_j|^2.
\end{equation}
Let
$\int_{Z_i} : \Omega^*_{L^2} \left(U_i; {\mathcal O}_i \right)
\rightarrow \Omega^{*-(n-n_\infty)}_{L^2} \left( (0,1)^{n_\infty} \right)$
denote fiberwise integration.

\begin{lemma}
  The form $\int_{Z_i} \chi_i \wedge \sigma_i \in
  \Omega^{p}_{L^2} \left( (0,1)^{n_\infty} \right)$
  lies in $\Ker(d_{(0,1)^{n_\infty}})$.
\end{lemma}
\begin{proof}
Given a compactly supported
$\omega^\prime \in \Omega^{n_\infty-p-1}_{\Lip} \left( (0,1)^{n_\infty} \right)$, 
we have
\begin{equation} \label{newnewint}
   \int_{(0,1)^{n_\infty}} d\omega^\prime \wedge \int_{Z_i} \chi_i \wedge
  \sigma_i
=
\int_{U_i} d h_i^* \omega^\prime \wedge \chi_i \wedge \sigma_i.
\end{equation}
Smoothing $\sigma_i$ by applying the heat operator on $M_i$, and using
Lemma \ref{newlemma}, gives
\begin{equation}
\int_{U_i} d h_i^* \omega^\prime \wedge \chi_i \wedge \sigma_i =
\int_{U_i} d \left( h_i^*\omega^\prime \wedge \chi_i \wedge \sigma_i
\right) = 0.
\end{equation}
The lemma follows.
\end{proof}

Put $\nu_i = \int_{Z_i} \chi_i \wedge \sigma_i$.
As
\begin{equation}
  \int_{Z_i} \chi_i \wedge \eta_i =
  \int_{Z_i} \chi_i \wedge h_i^* \sum_{j=1}^{\overline{k}_i} \beta_j \psi_j =
  \sum_{j=1}^{\overline{k}_i} \beta_j \psi_j,
\end{equation}
it follows that
\begin{equation}
  \sum_{j=1}^{\overline{k}_i} \beta_j \psi_j + \nu_i = 
  \int_{Z_i} \chi_i \wedge (\eta_i + \sigma_i). 
\end{equation}

Using Lemma \ref{chi} and (\ref{needed}),
\begin{align}
&  \int_{(0,1)^n} \left\langle \sum_{j=1}^{\overline{k}_i}
  \beta_j \psi_j + \nu_i,
  \sum_{j=1}^{\overline{k}_i}
\beta_j \psi_j + \nu_i \right\rangle 
\: \dvol_{(0,1)^n} = \\
&  \int_{(0,1)^n} \left\langle 
\int_{Z_i} \chi_i \wedge (\eta_i + \sigma_i),
\int_{Z_i} \chi_i \wedge (\eta_i + \sigma_i) \right\rangle 
\: \dvol_{(0,1)^n} \le \notag \\
& \Delta
\left( \inf_{u \in (0,1)^{n_\infty}} \vol(h_i^{-1}(u))
  \right)^{-1}
\int_{(0,1)^n}
\int_{h_i^{-1}(u)} \langle \eta_i + \sigma_i,
\eta_i + \sigma_i \rangle \dvol_{h_i^{-1}(u)} \: \dvol_{(0,1)^{n_\infty}}(u) < \notag \\
& \frac12 2^{\frac{n!}{n_\infty}c_i(2p-n_\infty)} N_\psi
  \sum_{j=1}^{\overline{k}_i} |\beta_j|^2. \notag 
\end{align}
However,
\begin{align}
&  \int_{(0,1)^n} \left\langle \sum_{j=1}^{\overline{k}_i}
  \beta_j \psi_j + \nu_i,
  \sum_{j=1}^{\overline{k}_i}
\beta_j \psi_j + \nu_i \right\rangle 
\: \dvol_{(0,1)^n} = \\
&  \sum_{j=1}^{\overline{k}_i}
  \int_{B_j} \left\langle
  \beta_j \psi_j + \nu_i,
  \beta_j \psi_j + \nu_i \right\rangle 
  \: \dvol_{B_j} \ge 2^{\frac{n!}{n_\infty}c_i(2p-n_\infty)} N_\psi
  \sum_{j=1}^{\overline{k}_i} |\beta_j|^2 . \notag
\end{align}
This is a contradiction.

To finish the proposition, combining Lemma \ref{lem4.16}, (\ref{upper}) and (\ref{lower}) gives
\begin{equation}
  \lambda_{\overline{k}_i,p} \le 8 \cdot 2^{\frac{2(n!)}{n_\infty} c_i} \Lambda^{p+1} \Delta \widehat{C} E_\psi
  N_{\psi}^{-1} = 
8 \Lambda^{p+1} \Delta \widehat{C} E_\psi
N_{\psi}^{-1}  \overline{k}_i^{\frac{2}{n_\infty}} \le
8 \Lambda^{p+1} \Delta \widehat{C} E_\psi
N_{\psi}^{-1}  \overline{k}_i^{\frac{2}{s}}
\end{equation} 
For large $i$, this contradicts the fact that
$\lambda_{\overline{k}_i,p}(M_i)
\ge \frac{i}{4^{\frac{n!}{s}}}  \overline{k}_i^{\frac{2}{s}}$.
\end{proof}
  
Theorem \ref{thm1.3} follows from Proposition \ref{prop4.13} and the
Hodge decomposition.

\begin{remark} \label{1.12}
  In addition to upper bounds on the eigenvalues, one could ask about
  lower bounds.
  Under the assumption of a lower bound on the curvature operator, there
  are lower bounds on $\lambda_{k,p}$ in terms of the diameter,
  coming from heat trace estimates, with the
  bound proportionate to $k^{\frac{2}{n}}$ as $k \rightarrow \infty$
  \cite{Berard (1988)}. (For some finite number of $k$'s, the
  lower bound may be zero.) 
  If we only assume a lower bound on the sectional curvature then it is
  not so clear if there are eigenvalue bounds from below, in terms of the
  diameter.  As a consistency
  check, we note that eigenvalue bounds from below, going to infinity
  as $k \rightarrow \infty$, imply upper bounds on
  Betti numbers.  Such upper bounds on Betti numbers exist if
  we assume a lower bound on the
  curvature operator or, more generally a lower bound on sectional curvatures.
  However, the proofs are
  very different in the two cases.
  With a lower bound on the curvature operator, a Betti number bound
  (in terms of the diameter) follows easily from heat trace estimates.
  In comparison, if we assume a lower bound on sectional curvatures
  then there is again
  a Betti number bound, but the proof uses completely different methods
  \cite{Gromov (1981)}.

  One can also ask about spectral convergence.
  That is, suppose that $\{(M_i, g_i)\}_{i=1}^\infty$ is a sequence
  of Riemannian manifolds satisfying the assumptions of Theorem
  \ref{thm1.7}, with a Gromov-Hausdorff limit $X$. The question is
  whether $\lim_{i \rightarrow \infty} \lambda_{k,p}(M_i) =
  \lambda_{k,p}(X)$. This is known for functions
  when the lower sectional curvature bound is
  replaced by a lower Ricci curvature
  bound \cite{Cheeger-Colding (2000)}, and for $1$-forms when the
  lower sectional curvature bound is replaced by a double sided Ricci bound
  \cite{Honda (2015)}. It may be necessary to assume that the Riemannian
  manifolds $\{(M_i, g_i)\}_{i=1}^\infty$ have a uniform lower bound
  on the curvature operator.
\end{remark}

\section{Hodge theorem and compact resolvent} \label{sec5}

In this section we introduce the class ${\mathcal C}_*$ of
Lipschitz multiconical spaces. 
If $X \in {\mathcal C}_*$
then in Subsection \ref{subsec5.1} we prove a Hodge theorem,
in the sense that we identify $\Ker(\triangle_*)$ with a certain
intersection homology group. In Subsection \ref{subsec5.2} we show that
$(I + \triangle_*)^{-1}$ is compact.  

To begin,
if $Y$ is a metric space of diameter at most $\pi$, and $\epsilon > 0$, then
the truncated open metric cone $CY(\epsilon)$ over $Y$ is
homeomorphic to the topological space $([0, \epsilon) \times Y)/\sim$,
where $(0, y_1) \sim (0,y_2)$ for all $y_1, y_2 \in Y$. The vertex of the
cone, i.e. the equivalence class $\{(0,y)\}_{y \in Y}$, is denoted by
$\star$.  The metric on $CY$ comes from
\begin{equation} \label{5.1}
d_{CY} ((t_1,y_1),(t_2,y_2)) = t_1^2 + t_2^2 - 2 t_1 t_2
\cos(d_Y(y_1,y_2)).
\end{equation}

We define a class ${\mathcal C}_*$ of compact metric spaces inductively.
An element of ${\mathcal C}_0$ is a finite metric space.
A compact
metric space $X$ lies in ${\mathcal C}_n$ if every point $x \in X$
has a neighborhood $U$ so that there is a pointed biLipschitz
homeomorphism $h : (U,x) \rightarrow (CY(\epsilon), \star)$ for
some $\epsilon > 0$ and 
some $Y \in {\mathcal C}_{n-1}$ with $\diam(Y) \le \pi$.

Note that for any $\epsilon > 0$, the cone
$CY(\epsilon)$ is biLipschitz homeomorphic to 
$CY(1)$, so the parameter $\epsilon$ is not really needed in the
definition.

\begin{example}
  If $X$ is an $n$-dimensional closed Riemannian Lipschitz
  manifold then $X \in {\mathcal C}_n$.
  \end{example}

By induction, if $X \in {\mathcal C}_n$ then there is an open
dense subset $X^*$ of full Hausdorff $n$-measure
that has the structure of a  Riemannian Lipschitz
$n$-manifold.  Namely if $n = 0$ then $X^* = X$. If $n > 0$,
cover $X$ by a finite number of neighborhoods $\{U_i\}_{i=1}^N$
of points $\{x_i\}_{i=1}^N$ so that there are pointed
biLipschitz homeomorphisms $\phi_i : (U_i, x_i)
\rightarrow (CY_i(\epsilon_i), \star_i)$, with $Y_i \in
{\mathcal C}_{n-1}$. Then we can take
$X^* = \bigcup_{i=1}^N \phi_i^{-1} (CY_i^*(\epsilon_i) - \star_i)$.

We can apply the setup of Section \ref{sec2}.

\begin{example} \label{ex5.2}
An element $X$ of ${\mathcal C}_1$ is a finite metric graph $G$.
The subset $X^*$ is the union of the open edges.
For convenience, suppose that $X^*$ is oriented.
Given $f \in \Omega^0_{L^2}(X)$ and an edge $e$ of length $L_e$, write
$f \Big|_e = f_e$, where $f_e \in L^2(0, L_e)$. Then $f \in \Dom(d)$ if
\begin{enumerate}
\item Each $f_e \in H^1(0, L_e)$, and
\item
For each vertex $v$ of $G$, the 
sum of the limiting values of $f_e$ along edges $e$ incoming to $v$
equals the sum of the limiting values of $f_e$ along edges outgoing from $v$.
\end{enumerate}

Note that $\Omega^0_{L^2}(X)/\overline{\Image(d)} =
\Omega^0_{L^2}(X)$.
A $1$-form $\omega \in \Omega^1_{L^2,d}(X) =
\Omega^1_{L^2}(X)$ consists of a union of
$L^2$-regular $1$-forms on the open edges.
The $L^2$-cohomology of $X$ is given in degree zero by
$\Ker(d : \Omega^0_{L^2,d}(X) \rightarrow \Omega^1_{L^2,d}(X))
\cong \R^{b_1(G)}$,
and in degree one by $\Omega^1_{L^2,d}(X)/{\Image(d)} \cong
\R^{b_0(G)}$.

Given $\omega \in \Omega^1_{L^2}(X)$ and an edge $e$ of length $L_e$,
write $\omega \Big|_e = {\omega}_e ds$, where
${\omega}_e \in L^2(0, L_e)$ and $s$ is the oriented
length parameter along $e$. Then $\omega \in \Dom(d^*)$ if
\begin{enumerate}
\item Each $\omega_e$ lies in $H^1(0, L_e)$, and
\item
For each vertex $v \in G$, there is a number $F_v$ so that for
each edge $e$ adjoining
$v$, the limiting value of $\omega_e$ on $e$, toward $v$, is
$\pm F_v$, depending on whether $e$ is incoming or outgoing.
\end{enumerate}
If $\omega \in \Dom(d^*)$ then the restriction of
$d^* \omega \in \Omega^0_{L^2}(X)$ to an edge $e$ is
$- \: \frac{d\omega_e}{ds}$.

Then $\Dom(\triangle_0) =
\{f \in \Omega^0_{L^2}(X)\: : \: df \in \Dom(d^*)\}$.
The restriction of $\triangle_0 f$ to an edge $e$ is
$- \: \frac{d^2 f_e}{ds^2}$. The operator $\triangle_1$ on
$\Omega^1_{L^2}(X)/\overline{\Image(d)}$ vanishes.

To see how the orientation of $G$ affects the calculations,
suppose that $G_e^\prime$ is the oriented graph obtained by
starting with $G$ and reversing
the orientation of a particular edge
$e$. Given $f$ in the domain of $d$ for $G$,
define
$f^\prime$ by
\begin{equation} \label{5.3}
  f^\prime \Big|_{e^\prime} =
\begin{cases}
-  f \Big|_{e^\prime}, & \mbox{if } e^\prime = e \\
f \Big|_{e^\prime}, & \mbox{if } e^\prime \neq e.
\end{cases}
\end{equation}
Then $f^\prime$ is in the domain of $d$ for $G^\prime$, and
has the same energy as $f$. Hence if $f$ is in the
domain of $\triangle_0$ for $G$, then 
$f^\prime$ is in the
domain of $\triangle_0$ for $G^\prime$. It follows that choosing
different orientations of the graph gives unitarily equivalent
representations of $\triangle_0$.
\end{example}

\subsection{Hodge theorem} \label{subsec5.1}

For background information on intersection homology, we refer to
\cite{Kirwan (2006)}.

\begin{proposition} \label{prop5.10}
  If $X \in {\mathcal C}_n$ then for all $p \in [0,n]$, the unreduced
  $L^2$-cohomology
\begin{equation}
  {\mathcal H}_{L^2}^p(X) = \Ker \left( d : \Omega^p_{L^2,d}(X) \rightarrow
  \Omega^{p+1}_{L^2,d}(X)
  \right)/
\Image \left( d : \Omega^{p-1}_{L^2,d}(X) \rightarrow \Omega^{p}_{L^2,d}(X)
\right)
\end{equation}
is isomorphic to
$\IH^{GM}_{n-p}(X; {\mathcal O})$. Here
$\IH^{GM}_*$ denotes the Goresky-MacPherson intersection homology
with real coefficients, computed with the perversity $\overline{p}$ given by
$\overline{p}(0) = 0$ and
$\overline{p}(j) = \left[ \frac{j-1}{2} \right]$ for $j \ge 1$,
and ${\mathcal O}$ is the orientation line bundle of the
codimension-zero stratum of $X$.
\end{proposition}
\begin{proof}
We first note that $\IH$ is a topological invariant of $X$
\cite[Theorem 9]{King (1985)} but can be computed using a
topological stratification.
In our case there is a natural stratification
$X = X_n \supset X_{n-1} \supset \ldots \supset X_0 \supset
X_{-1} = \emptyset$ given by saying
that $x \in X_j$ if and only if there is no neighborhood
of $x$ that is biLipschitz equivalent to $B^{j+1} \times CY^{n-j-2}$
for any $Y^{n-j-2} \in {\mathcal C}_{n-j-2}$ of diameter at most $\pi$.
Here $B^{j+1}$ is the open unit ball in $\R^{j+1}$.
The associated
codimension-$k$ stratum is $X_{n-k} - X_{n-k-1}$, a manifold
of dimension $n-k$. A point in the codimension-$k$
stratum has a neighborhood that splits off a $B^k$-factor, but there is
no neighborhood that splits off a $B^{k+1}$-factor.

If ${\mathcal O}$ is the orientation line bundle of
the codimension-zero stratum then since $\overline{p}(1) = 0$,
it follows as in \cite[Section 2.2]{Goresky-MacPherson (1983)}
that perversity-$\overline{p}$
intersection homology with values in ${\mathcal O}$
is well-defined. 

If $X$ is a pseudomanifold, i.e. if $X_{n-1} = X_{n-2}$,
then the result
of Proposition \ref{prop5.10} is well known and goes back to work of
Cheeger \cite{Cheeger (1983)} and
Cheeger-Goresky-MacPherson \cite{Cheeger-Goresky-MacPherson (1982)}.
The proofs that are relevant for us are sheaf-theoretic in nature, and
appear in \cite{Nagase (1986)} and \cite{Youssin (1994)}.

There are two relevant sheaves of differential graded
complexes on $X$. The first one is
the sheaf $\Omega^*_{L^2_{loc},d}$ defined in Section \ref{sec3}.
We will use the fact that it only depends on $X$ through the
biLipschitz homeomorphism class of $X$.
The second relevant sheaf, as pointed out to me by Greg Friedman, is
the sheaf $\IC_{n-\star}$ coming from the presheaf
whose sections, over an open set
$U \subset X$, are the singular intersection chains
$C_{n-\star}(X, X - \overline{U}; {\mathcal O})$ relative to
the perversity $\overline{p}$
\cite[Section 3]{Friedman (2007)}.
(If $X$ is a PL-stratified space then $IC_{n-\star}$ is derived
isomorphic to the sheaf $IC^\infty_{n-\star}$ whose sections,
over an open set $U$, are the locally finite
${\mathcal O}$-valued simplicial intersection chains
relative to $\overline{p}$.)
The hypercohomology groups
of the two sheaves
are isomorphic to ${\mathcal H}^*_{L^2}(X)$
and $\IH^{GM}_{n-*}(X; {\mathcal O})$,
respectively.  Hence it suffices to show that the two sheaves
are isomorphic in the derived category of differential graded
sheaves on $X$.

When $X$ is a pseudomanifold, the strategy of
\cite{Nagase (1986)} was to use the unique extension result
of \cite{Goresky-MacPherson (1983)}. On the codimension-zero
stratum, each sheaf was quasi-isomorphic to the constant $\R$-sheaf
in degree $0$.  As each sheaf satisfied the
axioms of \cite[Section 3.4]{Goresky-MacPherson (1983)}, the stratum-by-stratum
argument of 
\cite[Section 3.5]{Goresky-MacPherson (1983)} showed that the two extensions
from the codimension-zero stratum
to all of $X$ are
isomorphic in the derived category.

When the codimension-one stratum is nonempty, this strategy has
to be slightly modified. The relevant unique extension result for us is
in \cite[Section 4]{Habegger-Saper (1991)}, with
$c_p = c =2$. We have to know that the
restrictions of
the two sheaves, on the union of the
codimension-zero and codimension-one strata, are equivalent
coefficient systems in the sense of
\cite[Definition 5.1]{Habegger-Saper (1991)}. We also have
to know that the conditions of
\cite[Proposition 5.2(1)]{Habegger-Saper (1991)} are satisfied.
Then \cite[Proposition 4.5]{Habegger-Saper (1991)} implies that
the two sheaves are isomorphic in the derived category of
differential graded sheaves on $X$.

As the steps are similar to those in the pseudomanifold case,
we just give the main points. We let
$\HH^*_{L^2,d}(\cdot)$ denote hypercohomology of the sheaf
$\Omega^*_{L^2_{loc},d}$ and we let
$\HH^*_{L^2,d,c}(\cdot)$ denote compactly-supported
hypercohomology.
Similarly, we let $\IC^{GM}_{n-\star}(\cdot)$ denote
hypercohomology of the sheaf $\IC_{n-\star}$ and we let
$\IC^{GM}_{n-\star,c}(\cdot)$ denote compactly supported
hypercohomology, i.e the usual ${\mathcal O}$-twisted
Goresky-MacPherson intersection
homology in degree $n - \star$. 

First, we give the relevant cohomology of the truncated open metric cone
$CY = CY(1)$ over some $Y \in {\mathcal C}_{k-1}$.
If $k=1$ then
\begin{align} \label{5.11}
\HH^0_{L^2,d}(CY) = \IH^{GM}_1(CY) & = \widetilde{\HH}_0(Y), \\
\HH^1_{L^2,d}(CY) = \IH^{GM}_0(CY) & = 0, \notag
\end{align}
where $\widetilde{\HH}$ denotes reduced homology,
and
\begin{align} \label{5.12}
\HH^0_{L^2,d,c}(CY) = \IH^{GM}_{1,c}(CY) & =0, \\
\HH^1_{L^2,d,c}(CY) = \IH^{GM}_{0,c}(CY) & = \R. \notag
\end{align}
If $k > 1$ then
\begin{equation} \label{5.13}
  \HH^i_{L^2,d}(CY) \cong
  \begin{cases}
    \HH^i_{L^2,d}(Y) & \mbox{ if } i < \frac{k}{2}, \\
    0 & \mbox{ if } i \ge \frac{k}{2},
  \end{cases}
  \end{equation}

\begin{equation} \label{5.14}
  \HH^i_{L^2,d,c}(CY) \cong
  \begin{cases}
    \HH^{i-1}_{L^2,d}(Y) & \mbox{ if } i \ge \frac{k}{2} + 1, \\
    0 & \mbox{ if } i < \frac{k}{2} + 1,
  \end{cases}
  \end{equation}

\begin{equation} \label{5.15}
  \IH^{GM}_{k-i}(CY) \cong
  \begin{cases}
    \IH^{GM}_{k-1-i}(Y) & \mbox{ if } i < \frac{k}{2}, \\
    0 & \mbox{ if } i \ge \frac{k}{2}
  \end{cases}
  \end{equation}
  and
\begin{equation} \label{5.16}
  \IH^{GM}_{k-i,c}(CY) \cong
  \begin{cases}
    \IH^{GM}_{k-i,c}(Y) & \mbox{ if } i \ge \frac{k}{2} + 1, \\
    0 & \mbox{ if } i < \frac{k}{2} + 1.
  \end{cases}
\end{equation}
Equations (\ref{5.13}) and
(\ref{5.14}) can be proved using separation of variables as in
\cite{Youssin (1994)}.
Equation (\ref{5.13}) can be understood as saying that
the $L^2$-cohomology of $CY$ comes from pulling back
harmonic forms from $Y$ with respect to the projection map
$(0,1) \times Y \rightarrow Y$, provided that the pullback
is square integrable.
Equation (\ref{5.14}) can be understood as saying that
the compactly supported $L^2$-cohomology of $CY$ is generated by
forms of the type $\phi^\prime dr \wedge \omega$, where
$\phi \in C^\infty(0,1)$ is a nonincreasing function that is identically
one on $(0, 1/3)$ and identically zero on $(2/3, 1)$, and
$\omega$ is a harmonic $(i-1)$-form on $Y$.
Then when $i \ge \frac{k}{2} + 1$, the putative primitive
$\phi \omega$ fails to be square integrable on $CY$.
Equations (\ref{5.15}) and (\ref{5.16}) follow from
\cite[Proposition 2.18]{Friedman (2007)}.

Next, for both sheaf cohomology theories there are K\"unneth
formulas  :
\begin{align} \label{5.17}
  \HH^i_{L^2,d}(B^{n-k} \times CY) & \cong
    \HH^i_{L^2,d}(CY), \\
  \HH^i_{L^2,d,c}(B^{n-k} \times CY) & \cong
    \HH^{i-n+k}_{L^2,d,c}(CY), \notag \\
  \IH^{GM}_{n-i}(B^{n-k} \times CY) & \cong
    \IH^{GM}_{k-i}(CY), \notag \\
  \IH^{GM}_{n-i,c}(B^{n-k} \times CY) & \cong
    \IH^{GM}_{n-i,c}(CY). \notag
\end{align}
The third isomorphism comes from
\cite[Proposition 2.20]{Friedman (2007)}.

To check the conditions of \cite{Habegger-Saper (1991)},
let $A$ denote one of the two above differential graded sheaves.
They are both cohomologically constructible with respect to
the stratification. As in
\cite[Section 1.12]{Goresky-MacPherson (1983)}, 
if $v$ is the vertex of a cone $CY$, and $f : (0,v)
\rightarrow B^{n-k} \times CY$ is inclusion, then
$\HH^i(f^* A) \cong \HH^i(B^{n-k} \times CY; A)$ and
$\HH^i(f^! A) \cong \HH_c^i(B^{n-k} \times CY; A)$.
We will use the fact that the cohomology of
$\Omega^*_{L^2_{loc}, d}$ is biLipschitz invariant, in order
to compute the $L^2$-cohomology of a neighborhood
of a point $x \in X$ that is biLipschitz to
$B^{n-k} \times CY$, using the conical metric on $CY$ and the product metric
on $B^{n-k} \times CY$.

The restrictions of $\Omega^*_{L^2_{loc},d}$ and $\IC_{n-\star}$ to
$X_n - X_{n-2}$ are quasi-isomorphic; c.f. (\ref{5.11}) and (\ref{5.12}).
Let ${\mathcal E}$ denote their common class in the derived
category of differential graded sheaves on $X_n - X_{n-2}$.
From (\ref{5.11}), (\ref{5.12}) and (\ref{5.17}), if $x \in X_n - X_{n-2}$ then
$\HH^i({\mathcal E}_x) = 0$ for $i > 0$ and
$\HH^i(f_x^!{\mathcal E}) = 0$ for $i < n$. 
It follows from \cite[Proposition 5.2]{Habegger-Saper (1991)} that
the coefficient system satisfies
\cite[Definition 5.1]{Habegger-Saper (1991)}.

Finally, from (\ref{5.13}), (\ref{5.15}) and (\ref{5.17}), we have
\begin{equation} \label{5.18}
\HH^i_{L^2,d}(B^{n-k} \times CY) =
\HH^{GM}_{n-i}(B^{n-k} \times CY) = 0
\end{equation}
if $Y \in {\mathcal C}_{k-1}$ and 
$i > \overline{p}(k)$. 
From (\ref{5.14}), (\ref{5.16}) and (\ref{5.17}), we have
\begin{equation} \label{5.19}
\HH^i_{L^2,d,c}(B^{n-k} \times CY) =
\HH^{GM}_{n-i,c}(B^{n-k} \times CY) = 0
\end{equation}
if  $Y \in {\mathcal C}_{k-1}$ and
$i <  n - \max(n-2-\overline{p}(k), 0)$.
Using \cite[Lemma 4.6]{Habegger-Saper (1991)}, the sheaves
$\Omega^*_{L^2_{loc},d}$ and $\IC_{n-\star}$ satisfy the axioms
$A_{\overline{p}}$ in the sense of
\cite[Definition 4.3]{Habegger-Saper (1991)}.
From \cite[Remark 4.2]{Habegger-Saper (1991)} and
\cite[Proposition 4.5]{Habegger-Saper (1991)}, it follows that the two sheaves
are derived isomorphic on $X$.  This proves the proposition.
\end{proof}

\begin{corollary} \label{cor5.8}
If $X \in {\mathcal C}_n$ then for all $p \in [0,n]$,
we have
$\dim(\Ker(\triangle_p)) < \infty$.
\end{corollary}

\begin{corollary} \label{cor5.9}
If $X \in {\mathcal C}_n$ then for all $p \in [0,n]$,
we have
$\overline{\Image(d)} = \Image(d) \subset \Omega^p_{L^2}(X)$.
\end{corollary}

In particular, $\Ker(\triangle_*) \cong {\mathcal H}^*_{L^2}(X)$.

\subsection{Compactness of the resolvent} \label{subsec5.2}

\begin{proposition} \label{prop5.4}
For any $X \in {\mathcal C}_n$, and for any
$p \in [0, n]$, the operator $(I + \triangle_p)^{-1}$
is a compact operator on $\Omega^p_{L^2}(X)/\overline{\Image(d)}$.
\end{proposition}
\begin{proof}
We will prove the proposition by induction on $n$.  It is true
if $n=0$.

Using Corollary \ref{cor5.9},
we have
\begin{equation}
\Omega^p_{L^2,d}(X)/{\Image(d)} =
\Omega^p_{L^2,d}(X)/\overline{\Image(d)} \cong
{\mathcal H}^p_{L^2}(X) \oplus
\left( \Omega^p_{L^2,d}(X)/\Ker(d) \right).
\end{equation}
The restriction of $I + \triangle_p$ to
$\Omega^p_{L^2,d}(X)/\Ker(d)$ equals $I + d^* d$.
Let $\Omega^p_{H^1,d}(X)/\Ker(d) \subset \Omega^p_{L^2}(X)/\Ker(d)$
denote $\Omega^p_{L^2,d}(X)/\Ker(d)$ with the Hilbert space
norm $\| \omega \|_{H^1}^2 =
\| d\omega \|_{L^2}^2 + \| \omega \|_{L^2}^2$.
The map $d : \Omega^p_{H^1,d}(X)/\Ker(d) \rightarrow
\Image(d)$ is invertible. Since $\Image(d)$ is closed in
$\Omega^{p+1}_{L^2}(X)$,
the open mapping theorem implies that the inverse
$d^{-1} : \Image(d) \rightarrow \Omega^p_{H^1,d}(X)/\Ker(d)
\subset \Omega^p_{L^2}(X)/\Ker(d)$ is $L^2$-bounded.
As $(d^* d)^{-1} = d^{-1} \left( d^{-1} \right)^*$, it suffices to show
that $d^{-1}$ is compact.

We will construct an approximate inverse of $d$.
Given $x \in X$, let $V_x$ be a neighborhood of $x$ for which there is a
biLipschitz homeomorphism
$\phi_x : (V_x, x) \rightarrow (CY_x(2), \star_x)$
with $Y_x \in {\mathcal C}_{n-1}$ of diameter at most $\pi$.
(Note that the parameter $2$ in $CY_x(2)$ can always be chosen.)
Put $U_x = \phi_x^{-1}(CY_x(1))$.
As $\{U_x\}_{x \in X}$ is an open covering of $X$, we can choose a
finite subcovering
$\{U_i\}_{i=1}^N$, where we write $U_i = U_{x_i}$. 
We write the corresponding biLipschitz
homeomorphisms as $\phi_i \: : \: (V_i, x_i) \rightarrow (CY_i(2), \star_i)$.

Let $\Omega^{*}_{L^2}(V_i)$ be the square-integrable forms of $V_i$.
Let $\Omega^{*}_{L^2,d,abs}(V_i)$ be the square-integrable elements of
$\Omega^{*}_{L^2_{loc},d}(V_i)$, the latter being defined in
Section \ref{sec3}.
In terms of the conical coordinate $r$ on $V_i$, an element of
$\Omega^{*}_{L^2}(V_i)$
can be written as $\omega = \omega_0(r) + dr \wedge \omega_1(r)$, with
$\omega_0(r) \in \Omega^*_{L^2}(Y_i)$ and 
$\omega_1(r) \in \Omega^{*-1}_{L^2}(Y_i)$.
If $\omega \in \Omega^{*}_{L^2,d,abs}(V_i)$ then
$\omega_0(r) \in \Omega^*_{L^2,d}(Y_i)$,
$\omega_1(r) \in \Omega^{*-1}_{L^2,d}(Y_i)$, and
$d_{Y_i} \omega_0$ and
$\partial_r \omega_0 - d_{Y_i} \omega_1$
are  square-integrable on $V_i$.
Let $d_i : \Omega^*_{L^2,d,abs}(V_i) \rightarrow
\Omega^{*+1}_{L^2}(V_i)$ be the differential on $V_i$.
From (\ref{5.13}), the unreduced $L^2$-cohomology of $V_i$ is finite
dimensional and so $\Image(d_i)$ is closed in
$\Omega^{p+1}_{L^2}(V_i)$.
As before, the map $d_i : \Omega^p_{L^2,d,abs}(V_i)/\Ker(d_i)
\rightarrow \Image(d_i)$ has an inverse that extends to a
bounded map $d_i^{-1} : \Image(d_i) \rightarrow
  \Omega^p_{L^2}(V_i)/\Ker(d_i)$.

\begin{lemma} \label{lem5.6}
For each $i$, the operator $d_i^{-1}$ is compact.
\end{lemma}
\begin{proof}
Under a biLipschitz change of metric on $V_i$,
one obtains equivalent norms on $\Image(d_i)$ and
$\Omega^p_{L^2}(V_i)/\Ker(d_i)$.
Hence to prove the lemma,
we can replace $V_i$ by $CY_i(2)$ with its conical metric.

The quadratic form $Q_i$ on
$\Omega^p_{L^2,d_i,abs}(CY_i(2))/\Ker(d_i)$
is defined as in (\ref{2.11}).
The corresponding operator $d_i^* d_i$, densely defined
on $\Omega^p_{L^2}(CY_i(2))/\Ker(d_i)$, has a domain whose
elements satisfy absolute (Neumann) boundary conditions on
$Y_i = \partial \overline{CY_i(2)}$. 
The spectrum of $d_i^* d_i$ can
be explicitly computed using separation of variables.  If $Y_i$ is a
smooth manifold then this was done in
\cite{Hartmann-Spreafico (2017),Vertman (2009)}.
(Since we are not interested in enforcing Hodge duality,
subtleties about
ideal boundary conditions do not arise.) In our case,
by the induction assumption, the differential form Laplacian on
$Y_i$ has a discrete spectrum with finite multiplicities.
Then by the same separation of variable
argument, this is also true for $d_i^* d_i$ on $CY_i(2)$.
From the explicit spectral decomposition, one sees that
$d_i^{-1} = (d_i^* d_i)^{-1} d_i^*$ is a compact operator on
$\Image(d_i)$.
\end{proof}

There is a restriction map from $\Image(d) \subset \Omega^*_{L^2}(X)$
to $\Image(d_i) \subset \Omega^*_{L^2}(V_i)$.
Let $L_i : \Omega^*_{L^2}(V_i)/\Ker(d_i) \rightarrow
\Ker(d_i)^\perp$ be the lifting isomorphism. Let
$q : \Omega^*_{L^2}(X) \rightarrow \Omega^*_{L^2}(X)/\Ker(d)$
be the quotient map.
Let $\{\sigma_i\}_{i=1}^N$ be a Lipschitz partition of unity subordinate to
$\{U_i\}_{i=1}^N$.
Define a compact operator
$A : \Image(d) \rightarrow \Omega^p_{L^2}(X)/\Ker(d)$ by
\begin{equation}
  A\omega =
  q \sum_{i=1}^N \sigma_i L_i d_i^{-1} \left( \omega \Big|_{V_i} \right).
\end{equation}
Then
\begin{equation}
dA\omega =
\sum_{i=1}^N d\sigma_i \wedge L_i d_i^{-1} \left( \omega \Big|_{V_i} \right)
+ \omega.
\end{equation}
The operator $K$ given by
\begin{equation}
K\omega =
\sum_{i=1}^N d\sigma_i \wedge L_i d_i^{-1} \left( \omega \Big|_{V_i} \right)
\end{equation}
is compact.  As $dA = I+K$ and $d^{-1}$ is bounded, we see that
$d^{-1} = A - d^{-1} K$ is compact.
\end{proof}

This proves Theorem \ref{thm1.16}.


\begin{thebibliography}{10}

\bibitem{Berard (1988)} P. B\'erard, ``From vanishing theorems to
  estimating theorems: the Bochner technique revisited'',
  Bull. Amer. Math. Soc. 19, p. 371–406 (1988) 

\bibitem{Burago-Burago-Ivanov (2001)} D. Burago, Y. Burago and S. Ivanov,
\underline{Metric geometry}, Amer. Math. Soc., Providence (2001)
  
\bibitem{Burago-Gromov-Perelman (1992)} Y. Burago, M. Gromov and
G. Perelman, ``A.D. Alexandrov's spaces with curvature bounded from below'',
Russ. Math. Surv. 47, p. 1-58 (1992)

\bibitem{Cheeger (1983)} J. Cheeger, ``Spectral geometry of singular
  Riemannian spaces'', J. Diff. Geom. 18, p. 575-657 (1983)

\bibitem{Cheeger-Colding (2000)} J. Cheeger and T. Colding,
  ``The structure of spaces with Ricci curvature bounded below III'',
  J. Diff. Geom. 54, p. 37-74 (2000)
  
\bibitem{Cheeger-Goresky-MacPherson (1982)} J. Cheeger, M. Goresky and
  R. MacPherson, ``$L^2$-cohomology and intersection homology of
  singular algebraic varieties'', in
  \underline{Seminar on differential geometry}, Princeton University Press,
  Princeton, p. 303-340 (1982)

\bibitem{Cheng (1975)} S.-Y. Cheng, ``Eigenvalue comparison theorems and
  its geometric applications'', Math. Z. 143, p. 289-297 (1975)

\bibitem{DeCecco-Palmieri (1991)} G. De Cecco and G. Palmieri,
  ``Integral distance on a Lipschitz Riemannian manifold'',
  Math. Z. 207, p. 223–243 (1991) 
  
\bibitem{Dodziuk (1982)} J. Dodziuk, ``Eigenvalues of the Laplacian on
  forms'', Proc. Amer. Math. Soc. 85, p. 437- 443 (1982)
  
\bibitem{Friedman (2007)} G. Friedman,
  ``Singular chain intersection homology for traditional and
  super-perversities'',
  Trans. Amer. Math. Soc. 359, p. 1977-2019 (2007)

\bibitem{Friedman (2011)} G. Friedman,
  ``An introduction to intersection homology with general perversity
  functions'', in \underline{Topology of stratified spaces} ,
  Mathematical Sciences Research Institute Publications 58,
  Cambridge University Press, p. 177-222 (2011) 
  
\bibitem{Goresky-MacPherson (1983)} M. Goresky and R. MacPherson,
  ``Intersection homology II'', Inv. Math 71, p. 77-129 (1983)

\bibitem{Gromov (1981)} M. Gromov, ``Curvature, diameter and Betti numbers'',
  Comm. Math. Helv. 56,p. 179-195 (1981)
  
\bibitem{Habegger-Saper (1991)} N. Habegger and L. Saper,
  ``Intersection cohomology of cs-spaces and Zeeman's filtration'',
  Inv. Math. 105, p. 247-272 (1991)
  
\bibitem{Hartmann-Spreafico (2017)} L. Hartmann and M. Spreafico,
``The analytic torsion of the finite metric cone over a compact
manifold'', J. Math. Soc. Japan 69, p. 311-371 (2017)

\bibitem{Harvey-Searle (2016)} J. Harvey and C. Searle,
``Orientation and symmetries of Alexandrov spaces with applications in
positive curvature'', J. Geom. Anal. 27, p. 1636-1666 (2017) 

\bibitem{Honda (2015)} S. Honda, ``Spectral convergence under bounded
  Ricci curvature'', J. Funct. Anal. 273, p. 1577-1662 (2017)

\bibitem{Kapovitch (2007)} V. Kapovitch,
  ``Perelman's stability theorem'', in
  \underline{Surveys in differential geometry XI},
  International Press, Somerville, p. 103-136 (2007)

\bibitem{King (1985)} H. King, ``Topological invariance of intersection
homology without sheaves'', Top. and its Appl. 20, p. 149-160 (1985)

\bibitem{Kirwan  (2006)} F. Kirwan,
  \underline{An introduction to intersection homology theory},
    second edition, Chapman and Hall/CRC, Boca Raton (2006) 

\bibitem{Kuwae-Machigashira-Shioya (2001)}
  K. Kuwae, Y. Machigashira and T. Shioya,
  ``Sobolev spaces, Laplacian, and heat kernel on Alexandrov spaces'',
  Math. Z. 238, p. 269-316 (2001)
  (2001)

\bibitem{Lott (1996)} J. Lott
``The zero-in-the-spectrum question'', Enseign. Math. 42, p. 341-376 (1996)
  
\bibitem{Mitsuishi (2016)} A. Mitsuishi,
``Orientability and fundamental classes of Alexandrov spaces with
applications'', preprint, https://arxiv.org/abs/1610.08024 (2016)

\bibitem{Mitsuishi-Yamaguchi (2014)} A. Mitsuishi and T. Yamaguchi,
  ``Locally Lipschitz contractibility of Alexandrov spaces and its
  applications'', Pac. J. Math. 270, p. 393-421 (2014)
  
\bibitem{Nagase (1986)} M. Nagase, ``Sheaf theoretic $L^2$-cohomology'',
in \underline{Advanced Studies in Pure Math. 8}, North Holland, p.
273-279 (1986)
  
\bibitem{Perelman} G. Perelman, ``DC structure on Alexandrov space'',
preprint, posted at http://www.math.psu.edu/petrunin/papers/alexandrov/Cstructure.pdf

\bibitem{Perelman2} G. Perelman, ``Alexandrov's space with curvatures
  bounded from below II'', preprint, posted at
  http://www.math.psu.edu/petrunin/papers/alexandrov/perelmanASWCBFB2+.pdf

\bibitem{Reed-Simon (1980)} M. Reed and B. Simon,
\underline{Methods of modern mathematical physics I},
Academic Press, San Diego (1980)

\bibitem{Reed-Simon (1975)} M. Reed and B. Simon,
\underline{Methods of modern mathematical physics II},
Academic Press, San Diego (1975)

\bibitem{Reed-Simon (1978)} M. Reed and B. Simon,
\underline{Methods of modern mathematical physics IV},
Academic Press, San Diego (1978)

\bibitem{Shioya (2001)} T. Shioya,
  ``Convergence of Alexandrov spaces and spectrum of Laplacian'',
  J. Math. Soc. Japan 53, p. 1-15 (2001)

\bibitem{Siebenmann (1972)}  L. Siebenmann,
  ``Deformation of homeomorphisms on stratified sets I, II'', Comm. Math.
  Helv. 47, p. 123-136,137-163 (1972)

\bibitem{Teleman (1983)} N. Teleman,
``The index of signature operators on Lipschitz manifolds'',
Publ. Math. IHES 58, p. 39-78 (1983)

\bibitem{Vertman (2009)} B. Vertman,
``Analytic torsion of a bounded generalized cone'',
Comm. Math. Phys. 290, p. 813-860 (2009)  

\bibitem{Yamaguchi (1991)} T. Yamaguchi, ``Collapsing and pinching under
  a lower curvature bound'', Ann. Math. 133, p. 317-357 (1991)

\bibitem{Youssin (1994)} B. Youssin, ``$L^p$ cohomology of cones and
  horns'', J. Diff. Geom. 39, p. 559-603 (1994)

\end{thebibliography}
\end{document}